\RequirePackage{fix-cm}
\documentclass[smallcondensed]{svjour3}     
\smartqed  
\usepackage[parfill]{parskip}    
\usepackage{amsmath, amssymb, latexsym, enumerate, mathrsfs}
\usepackage{graphicx}
\usepackage{epstopdf}
\usepackage{subfig}

\usepackage{hyperref}
\usepackage{array}
\usepackage{times}
\usepackage{moreverb}
\usepackage{marvosym}
\DeclareMathOperator*{\arginf}{arg\,inf}

\DeclareMathOperator*{\essinf}{ess\,inf}
\DeclareMathOperator*{\esssup}{ess\,sup}
\newtheorem{assumption}[lemma]{Assumption}
\newcommand{\etal}{\mbox{et al.}}
\def\qed{\hfill $\Box$}
\journalname{}
\begin{document}

\title{On the dynamic consistency of hierarchical risk-averse decision problems
}

\titlerunning{On the dynamic consistency of hierarchical risk-averse decision problems}        

\author{Getachew K. Befekadu \and Eduardo L. Pasiliao
}

\institute{G. K. Befekadu~ ({\large\Letter}\negthinspace) \at
	          National Research Council, Air Force Research Laboratory \& Department of Industrial System Engineering, University of Florida - REEF, 1350 N. Poquito Rd, Shalimar, FL 32579, USA. \\
	          \email{gbefekadu@ufl.edu}           
	          \and
	          E. L. Pasiliao \at
	          Munitions Directorate, Air Force Research Laboratory, 101 West Eglin Blvd, Eglin AFB, FL 32542, USA. \\
	          \email{pasiliao@eglin.af.mil}           
}

\date{Received: October 24, 2016 / Accepted: date}

\maketitle

\begin{abstract}
In this paper, we consider a risk-averse decision problem for controlled-diffusion processes, with dynamic risk measures, in which there are two risk-averse decision makers (i.e., {\it leader} and {\it follower}) with different risk-averse related responsibilities and information. Moreover, we assume that there are two objectives that these decision makers are expected to achieve. That is, the first objective being of {\it stochastic controllability} type that describes an acceptable risk-exposure set vis-\'{a}-vis some uncertain future payoff, and while the {\it second one} is making sure the solution of a certain risk-related system equation has to stay always above a given continuous stochastic process, namely {\it obstacle}. In particular, we introduce multi-structure, time-consistent, dynamic risk measures induced from conditional $g$-expectations, where the latter are associated with the generator functionals of two backward-SDEs that implicitly take into account the above two objectives along with the given continuous obstacle process. Moreover, under certain conditions, we establish the existence of optimal hierarchical risk-averse solutions, in the sense of viscosity solutions, to the associated risk-averse dynamic programming equations that formalize the way in which both the {\it leader} and {\it follower} consistently choose their respective risk-averse decisions. \,Finally, we remark on the implication of our result in assessing the influence of the {\it leader'}s decisions on the risk-averseness of the {\it follower} in relation to the direction of {\it leader-follower} information flow.

\end{abstract}
\section{Introduction}	 \label{S1}
Let $\bigl(\Omega, \mathcal{F},\{\mathcal{F}_t \}_{t \ge 0}, \mathbb{P}\bigr)$ be a probability space, and let $\{B_t\}_{t \ge 0}$ be a $d$-dimensional standard Brownian motion, whose natural filtration, augmented by all $\mathbb{P}$-null sets, is denoted by $\{\mathcal{F}_t\}_{t \ge 0}$, so that it satisfies the {\it usual hypotheses} (e.g., see \cite{Pro90} or \cite{GikS72}). \,We consider the following controlled-diffusion process over a given finite-time horizon $T>0$
\begin{align}
d X_t^{u,v} = m\bigl(t, X_t^{u,v}, (u_t,v_t)\bigr) dt + \sigma\bigl(t, X_t^{u,v}, (u_t,v_t)\bigr)dB_t, \notag \\
X_0^{u,v}=x, \quad  0 \le t \le T, \label{Eq1.1} 
\end{align}
where
\begin{itemize}
\item $X_{\cdot}^{u,v}$ is an $\mathbb{R}^{d}$-valued controlled-diffusion process,
\item $(u_{\cdot},v_{\cdot})$ is a pair of $(U \times V)$-valued measurable decision processes such that for all $t > s$, $(B_t - B_s)$ is independent of $(u_r, v_r)$ for $r \le s$ (nonanticipativity condition) and
\begin{align*}
\mathbb{E} \int_{s}^{t} \vert u_{\tau}\vert^2 d\tau < \infty \quad \text{and} \quad \mathbb{E} \int_{s}^{t} \vert v_{\tau}\vert^2 d\tau < \infty, \quad \forall t \ge s, 
\end{align*}
with $U$ and $V$ are open compact sets in $\mathbb{R}^{d}$, with $U \cap V = \varnothing$, 
\item $m \colon [0, T] \times \mathbb{R}^d \times (U \times V) \rightarrow \mathbb{R}^{d}$ is uniformly Lipschitz, with bounded first derivative, and
\item $\sigma \colon [0, T] \times \mathbb{R}^{d} \times (U \times V) \rightarrow \mathbb{R}^{d \times d}$ is Lipschitz with the least eigenvalue of $\sigma\,\sigma^T$ uniformly 
bounded away from zero for all $(x, (u,v)) \in \mathbb{R}^{d} \times (U \times V)$ and $t \in [0, T]$, i.e., 
\begin{align*}
 \sigma(t, x, (u,v))\,\sigma^T(t, x, (u, v)) \succeq \lambda I_{d \times d}, &\quad \forall (x, (u,v)) \in \mathbb{R}^{d} \times (U \times V),\\
  & \quad \forall t \in [0, T],  
\end{align*}
for some $\lambda > 0$.
\end{itemize}

{\bf Notation}: Let us introduce the following spaces that will be useful later in the paper.
\begin{itemize}
\item $L^2\bigl(\Omega, \mathcal{F}_t, \mathbb{P}; \mathbb{R}^{d} \bigr)$ is the set of $\mathbb{R}^{d}$-valued $\mathcal{F}_t$-measurable random variables $\xi$ such that $\bigl\Vert \xi \bigr\Vert^2 = \mathbb{E}\bigl\{\bigl\vert \xi \bigr\vert^2  \bigr\}< \infty$;
\item $L^{\infty}\bigl(\Omega, \mathcal{F}_t, \mathbb{P}\bigr)$ is the set of $\mathbb{R}$-valued $\mathcal{F}_t$-measurable random variables $\xi$ such that $\bigl\Vert \xi \bigr\Vert = \essinf \bigl\vert \xi \bigr \vert < \infty$;
\item $\mathcal{S}^2\bigl(t, T; \mathbb{R}^{d} \bigr)$ is the set of $\mathbb{R}^{d}$-valued adapted processes $\bigl (\varphi_s\bigr)_{t \le s \le T}$ on $\Omega \times [t, T]$ such that $\bigl\Vert \varphi \bigr\Vert_{[t, T]}^2 = \mathbb{E}\bigl\{\sup_{t \le s \le T} \bigl\vert \varphi_s \bigr\vert^2  \bigr\}< \infty$;
\item $\mathcal{H}^2\bigl(t, T; \mathbb{R}^{d} \bigr)$ is the set of $\mathbb{R}^{d}$-valued progressively measurable processes $\bigl (\varphi_s\bigr)_{t \le s \le T}$ such that $\bigl\Vert \varphi \bigr\Vert_{[t, T]}^2 = \mathbb{E}\bigl\{ \int_t^T \bigl\vert \varphi_s \bigr \vert^2 ds  \bigr\}< \infty$.
\end{itemize}

In this paper, we consider a risk-averse decision problem for the above controlled-diffusion process, in which there are two hierarchical decision makers (i.e., {\it leader} and {\it follower} with differing risk-averse related responsibilities and information) choose their decisions from progressively measurable strategy sets. That is, the {\it leader'}s decision $u_{\cdot}$ is a $U$-valued measurable control process from
\begin{align}
\mathcal{U}_{[0,T]} \triangleq \Bigl\{u\colon [0,T] \times \Omega \rightarrow U \,\Bigl\vert  u \,\, &\text{is an} \,\, \bigl\{\mathcal{F}_t\bigr\}_{t\ge 0}\text{- adapted} \notag \\
& \quad \text{and}\,\, \mathbb{E} \int_{0}^{T} \vert u_t\vert^2 dt < \infty \Bigr\}, \label{Eq1.3} 
\end{align}
and while the {\it follower'}s decision $v_{\cdot}$ is a $V$-valued measurable control process from 
\begin{align}
\mathcal{V}_{[0,T]} \triangleq \Bigl\{v\colon [0,T] \times \Omega \rightarrow V \, \Bigl\vert v \,\, &\text{is an} \,\, \bigl\{\mathcal{F}_t\bigr\}_{t\ge 0}\text{- adapted} \notag \\
& \quad  \text{and}\,\, \mathbb{E} \int_{0}^{T} \vert v_t\vert^2 dt < \infty \Bigr\}. \label{Eq1.4} 
\end{align}

Furthermore, we consider the following two cost functionals that provide information about the accumulated risk-costs on the time interval $[0, T]$ w.r.t. the {\it leader} and {\it follower}, i.e., 
\begin{align}
 \text{\it leader's risk-cost:} \quad  \xi_{0,T}^l(u,v) = \int_0^T c_l\bigl(t, X_t^{u,v}, u_t\bigr) dt + \Psi_l(X_T), \label{Eq1.5} 
\end{align}
and
\begin{align}
\text{\it follower's risk-cost:} \quad  \xi_{0,T}^f(u,v) = \int_0^T c_f\bigl(t, X_t^{u,v}, v_t\bigr) dt + \Psi_f(X_T), \label{Eq1.6} 
\end{align}
where $c_l \colon [0,T] \times \mathbb{R}^d \times V \rightarrow \mathbb{R}$ and $c_f \colon [0,T] \times \mathbb{R}^d \times W \rightarrow \mathbb{R}$ are measurable functions; and  $\Psi_l\colon \mathbb{R}^d \rightarrow \mathbb{R}$ and $\Psi_f\colon \mathbb{R}^d \rightarrow \mathbb{R}$ are also assumed measurable functions. 

Here, we remark that the corresponding solution $X_t^{u,v}$ in \eqref{Eq1.1} depends on the admissible risk-averse decision pairs $(u_{\cdot}, v_{\cdot}) \in \mathcal{U}_{[0,T]} \otimes \mathcal{V}_{[0,T]}$); and it also depends on the initial condition $X_0^{u,v}=x$.  As a result of this, for any time-interval $[t, T]$, with $t \in [0, T]$, the accumulated risk-costs $ \xi_{t,T}^1$ and $\xi_{t,T}^2$ depend on the risk-averse decisions $(u_{\cdot}, v_{\cdot}) \in \mathcal{U}_{[t,T]} \otimes \mathcal{V}_{[t,T]}$.\footnote{For any $t \in [0, T]$, $\mathcal{U}_{[t,T]}$ and \,$\mathcal{V}_{[t,T]}$ denote the sets of $U$- and $V$-valued $\bigl\{\mathcal{F}_s^t\bigr\}_{s \ge t}$-adapted processes, respectively (see Definition~\ref{Df2}).}  Moreover, we also assume that $f$, $\sigma$, $c_l$, $c_f$, $\Psi_l$ and $\Psi_f$, for $p \ge 1$, satisfy the following growth conditions  
\begin{align}
\bigl\vert m\bigl(t, x, (u, v)) \bigr\vert &+ \bigl\vert \sigma\bigl(t, x, (u,v) \bigr) \bigr\vert + \bigl\vert c_l\bigl(t, x, u\bigr) \bigr\vert + \bigl\vert \Psi_l\bigl(x\bigr) \bigr\vert \notag \\
 & \quad \le K \bigl(1 + \bigl\vert x \bigr\vert^p + \bigl\vert u \bigr\vert + \bigl\vert v \bigr\vert \bigr) \label{Eq1.7}
\end{align}
and
\begin{align}
\bigl\vert m\bigl(t, x, (u, v)) \bigr\vert &+ \bigl\vert \sigma\bigl(t, x, (u,v) \bigr) \bigr\vert + \bigl\vert c_f\bigl(t, x, v\bigr) \bigr\vert + \bigl\vert \Psi_f\bigl(x\bigr) \bigr\vert \notag \\
 & \quad \le K \bigl(1 + \bigl\vert x \bigr\vert^p + \bigl\vert u \bigr\vert + \bigl\vert v \bigr\vert \bigr), \label{Eq1.8}
 \end{align}
 for all $\bigl(t, x, (u,v) \bigr) \in [0,T] \times \mathbb{R}^{d} \times (U \times V)$ and for some constant $K > 0$.

On the same probability space $\bigl(\Omega, \mathcal{F},\{\mathcal{F}_t \}_{t \ge 0}, \mathbb{P}\bigr)$, we consider the following backward stochastic differential equation (BSDE) 
\begin{align}
- d Y_t = g\bigl(t, Y_t, Z_t\bigr) dt - Z_tdB_t, \quad Y_T=\xi, \label{Eq1.9}
\end{align}
where the terminal value $Y_T=\xi$ belongs to $L^2\bigl(\Omega, \mathcal{F}_T, \mathbb{P}; \mathbb{R}\bigr)$ and the generator functional $g \colon \Omega \times [0, T] \times \mathbb{R} \times \mathbb{R}^{d} \rightarrow \mathbb{R}$, with property that $\bigl(g\bigl(t, y, z\bigr)\bigr)_{0 \le t \le T}$ is progressively measurable for each $(y, z) \in \mathbb{R} \times \mathbb{R}^{d}$. We also assume that $g$ satisfies the following assumption.

\begin{assumption} \label{AS1}~\\\vspace{-5.0mm} 
\begin{enumerate} [{\rm (\ref{AS1}.1)}]
\item $g$ is Lipschitz in $(y, z)$, i.e., there exists a constant $K > 0$ such that, $\mathbb{P}$-a.s., for any $t \in [0, T]$, $y_1, y_2 \in \mathbb{R}$ and $z_1, z_2 \in \mathbb{R}^d$ 
\begin{align*}
\bigl\vert g\bigl(t, y_1, z_1\bigr) - g\bigl(t, y_2, z_2\bigr) \bigr\vert \le K \bigl(\bigl\vert y_1 - y_2 \bigr\vert + \bigl\Vert z_1 - z_2 \bigr\Vert\bigr).
\end{align*}
\item $g\bigl(t, 0, 0\bigr) \in \mathcal{H}^2\bigl(t, T; \mathbb{R} \bigr)$.
\item $\mathbb{P}$-a.s., for all $t \in [0, T]$ and $y \in \mathbb{R}$, $g\bigl(t, y, 0\bigr) = 0$.
\end{enumerate}
\end{assumption}
Then, we state the following lemma, which is used to establish the existence of a unique adapted solution (e.g., see \cite{ParP90} for additional discussions).
\begin{lemma} \label{L1}
Suppose that Assumption~\ref{AS1} holds. Then, for any $\xi \in L^2\bigl(\Omega, \mathcal{F}_T, \mathbb{P}; \mathbb{R}\bigr)$, the BSDE in \eqref{Eq1.9}, with terminal condition $Y_T=\xi$, i.e.,
\begin{align}
 Y_t = \xi + \int_t^T g\bigl(s, Y_s, Z_s\bigr) ds - \int_t^TZ_s dB_s, \quad 0 \le t \le T  \label{Eq1.10}
\end{align}
has a unique adapted solution
\begin{align}
 \bigl(Y_t^{T,g,\xi}, Z_t^{T,g,\xi}\bigr)_{0 \le t \le T} \in  \mathcal{S}^2\bigl(0, T; \mathbb{R} \bigr) \times  \mathcal{H}^2\bigl(0, T; \mathbb{R}^{d} \bigr). \label{Eq1.11}
\end{align}
\end{lemma}

Moreover, we recall the following comparison result that will be useful later (e.g., see \cite{ParT99}).
\begin{theorem}[Comparison Theorem] \label{T1}
Given two generators $g_1$ and $g_2$ satisfying Assumption~\ref{AS1} and two terminal conditions $\xi_1, \xi_2 \in L^2\bigl(\Omega, \mathcal{F}_T, \mathbb{P}; \mathbb{R}\bigr)$. Let $\bigl(Y_t^1, Z_t^1\bigr)$ and $\bigl(Y_t^2, Z_t^2\bigr)$ be the solution pairs corresponding to $\bigl(\xi_1, g_1\bigr)$ and $\bigl(\xi_2, g_2\bigr)$, respectively. Then, we have 
\begin{enumerate} [(i)]
\item Monotonicity: If $\xi_1 > \xi_2$ and $g_1 > g_2$, $\mathbb{P}$-a.s., then $Y_t^1 > Y_t^2$, $\mathbb{P}$-a.s., for all $t \in [0, T]$;
\item Strictly Monotonicity: In addition to ($i$) above, if we assume that $\mathbb{P}\bigl(\xi_1 > \xi_2\bigr) > 0$, then $\mathbb{P}\bigl(Y_t^1 > Y_t^2\bigr) > 0$, for all $t \in [0, T]$.
\end{enumerate}
\end{theorem}
In the following, we give the definition for a dynamic risk measure that is associated with the generator of BSDE in \eqref{Eq1.9}.

\begin{definition} \label{Df1} 
For any $\xi \in L^2\bigl(\Omega, \mathcal{F}_T, \mathbb{P}; \mathbb{R}\bigr)$, let $\bigl(Y_t^{T,g,\xi}, Z_t^{T,g,\xi}\bigr)_{0 \le t \le T} \in  \mathcal{S}^2\bigl(0, T; \mathbb{R} \bigr) \times  \mathcal{H}^2\bigl(0, T; \mathbb{R}^{d} \bigr)$ be the unique solution for the BSDE in \eqref{Eq1.9} with terminal condition $Y_T=\xi$. Then, we define the dynamic risk measure $\rho_{t,T}^g$ of $\xi$ by\footnote{Here, we remark that, for any $t \in [0,T]$, the conditional $g$-expectation (denoted by $\mathcal{E}_g\bigl[\xi \vert \mathcal{F}_t\bigr]$) is also defined by
\begin{align*}
\mathcal{E}_g\bigl[\xi \vert \mathcal{F}_t\bigr] \triangleq Y_t^{T,g,\xi}.
\end{align*}} 
\begin{align}
\rho_{t,T}^g \bigl[\xi \bigr] \triangleq Y_t^{T,g,\xi}.  \label{Eq1.12}
\end{align}
\end{definition}

\begin{remark} \label{R1.1}
Note that such a risk measure is widely used for evaluating the risk of uncertain future outcomes, and also assisting with stipulating minimum interventions required by financial institutions for risk management (e.g., see \cite{ArtDEH99}, \cite{Pen04}, \cite{ElKPQ97}, \cite{FolS02}, \cite{DetS05} or \cite{CorHMP02} for related discussions). In Section~\ref{S2}, we use multi-structure, time-consistent, dynamic risk measures induced from conditional $g$-expectations, where the latter are associated with the generator functionals of two backward-SDEs that implicitly take into account the cost functionals of the {\it leader} and {\it follower} along with the given continuous obstacle process; and we provide a hierarchical framework for the risk-averse decision problem for the controlled-diffusion process.
\end{remark}

Moreover, if the generator functional $g$ satisfies Assumption~\ref{AS1}, then a family of time-consistent dynamic risk measures $\bigl\{\rho_{t,T}^g\bigr\}_{t \in [0,T]}$ has the following properties (see \cite{Pen04} for additional discussions).
 
\begin{property} \label{Pr1} ~ \\ \vspace{-5.0mm}
\begin{enumerate} [{\rm (p1)}]
\item {\it Convexity}: If $g$ is convex for every fixed $(t, \omega) \in [0, T] \times \Omega$, then for all $\xi_1, \xi_2 \in L^2\bigl(\Omega, \mathcal{F}_T, \mathbb{P}; \mathbb{R}\bigr)$ and for all $\lambda \in L^{\infty}\bigl(\Omega, \mathcal{F}_t, \mathbb{P}; \mathbb{R}\bigr)$ such that $0 \le \lambda \le 1$
\begin{align*}
\rho_{t,T}^g\bigl[\lambda \xi_1 + (1-\lambda)\xi_2 \bigr] \le \lambda \rho_{t,T}^g\bigl[\xi_1\bigr] + (1- \lambda) \rho_{t,T}^g\bigl[\xi_1\bigr];
\end{align*}
\item {\it Monotonicity}:  For $\xi_1, \xi_2 \in L^2\bigl(\Omega, \mathcal{F}_T, \mathbb{P}; \mathbb{R}\bigr)$ such that $\xi_1 > \xi_2$ $\mathbb{P}$-a.s., then 
\begin{align*}
\rho_{t,T}^g\bigl[\xi_1\bigr] > \rho_{t,T}^g\bigl[\xi_2\bigr], \quad \mathbb{P}{\text-a.s.};
\end{align*}
\item {\it Trans-invariance}: For all $\xi \in L^2\bigl(\Omega, \mathcal{F}_T, \mathbb{P}; \mathbb{R}\bigr)$ and $\nu \in L^2\bigl(\Omega, \mathcal{F}_t, \mathbb{P}; \mathbb{R}\bigr)$
\begin{align*}
\rho_{t,T}^g\bigl[\xi + \nu\bigr] = \rho_{t,T}^g\bigl[\xi\bigr] + \nu;
\end{align*}
\item {\it Positive-homogeneity}: For all $\xi \in L^2\bigl(\Omega, \mathcal{F}_T, \mathbb{P}; \mathbb{R}\bigr)$ and for all $\lambda \in L^{\infty}\bigl(\Omega, \mathcal{F}_t, \mathbb{P}; \mathbb{R}\bigr)$ such that $\lambda > 0$
\begin{align*}
\rho_{t,T}^g\bigl[\lambda \xi \bigr] = \lambda \rho_{t,T}^g\bigl[\xi\bigr];
\end{align*}
\item {\it Normalization}: $\rho_{t,T}^g\bigl[0\bigr] = 0$ for $t \in [0, T]$.
\end{enumerate}
\end{property}

\begin{remark} \label{R1.2}
Note that, since the seminal work of Artzner \etal~\cite{ArtDEH99}, there have been studies on axiomatic dynamic risk measures, coherency and consistency in the literature (e.g., see \cite{DetS05}, \cite{Pen04}, \cite{Ros06}, \cite{FolS02} or \cite{CorHMP02}). Particularly relevant for us are, time-consistent, dynamic risk measures, induced from conditional $g$-expectations associated with generator functionals of BSDEs satisfying the above properties {\rm (p1)}--{\rm (p5)}.
\end{remark}

Here it is worth mentioning that some interesting studies on the dynamic risk measures, based on the conditional $g$-expecations, have been reported in the literature (e,g. see \cite{Pen04}, \cite{CorHMP02} and \cite{Ros06} for establishing connection between the risk measures and the generator of BSDE; and see also \cite{RSta10} for characterizing the generator of BSDE according to different risk measures). Recently, the authors in \cite{Rus10} and \cite{BefVP15} have provided interesting results on the risk-averse decision problem for Markov decision processes, in discrete-time setting, and, respectively, a hierarchical risk-averse framework for controlled-diffusion processes. Note that the rationale behind our framework follows in some sense the settings of these papers. However, to our knowledge, the problem of dynamic consistent risk-aversion for controlled-diffusion processes has not been addressed in the context of hierarchical argument, and it is important because it provides a mathematical framework that shows how a such framework can be systematically used to obtain optimal risk-averse decisions.

The remainder of this paper is organized as follows. In Section~\ref{S2}, using the basic remarks made in Section~\ref{S1}, we state the hierarchical risk-averse decision problem for the controlled-diffusion process. In Section~\ref{S3}, we present our main results -- where we introduce a framework under which the {\it follower} is required to {\it respond optimally} to the risk-averse decision of the {\it leader} so as to achieve an overall consistent risk-averseness. Moreover, we establish the existence of optimal risk-averse solutions, in the sense of viscosity solutions, to the associated risk-averse dynamic programming equations. Finally, Section~\ref{S4} provides further remarks.

\section{The hierarchical risk-averse decision problem formulation} \label{S2}
In order to make our hierarchical formulation more precise, we further assume following.
\begin{assumption} \label{AS2}~\\\vspace{-5.0mm} 
\begin{enumerate} [{\rm (\ref{AS2}.1)}]
\item $g \colon [0, T] \times \Omega \times \mathbb{R} \times \mathbb{R}^d \rightarrow \mathbb{R}$ is a measurable function that satisfies Assumption~\ref{AS1},
\item $\xi^{Target}$ is a real-valued random variable from $L^2(\Omega, \mathcal{F}_T, \mathbb{P}; \mathbb{R})$,
\item $h \colon [0, T] \times \mathbb{R}^{d} \rightarrow \mathbb{R}$ is jointly continuous in $t$ and $x$; and satisfying
\begin{align}
 h(t, x) \le K(1 +\vert x \vert^p), \quad  (t, x) \in [0, T] \times \mathbb{R}^{d},  \label{Eq2.1}
\end{align}
$h(t, x) \le \Psi_l(x)$ for $x \in \mathbb{R}^{d}$ and for some constant $K >0$,
\item an ``obstacle" $\bigl\{L_t, 0 \le t \le T \bigr\}$, which is continuous progressively real-valued process satisfying
\begin{align}
 \mathbb{E}\Bigl\{ \sup_{0 \le t \le T} (L_t^{+})^2 \Bigr\} < \infty.  \label{Eq2.2}
\end{align}
\end{enumerate}
\end{assumption}
Then, for any $(t, x) \in [0, T] \times \mathbb{R}^d$, we consider the following forward-SDE with an initial condition $X_t^{t,x;w} = x$
\begin{align}
d X_s^{t,x;w} = m\bigl(t, X_s^{t,x;w}, (u_s, v_s)\bigr) ds + \sigma\bigl(s, X_s^{t,x;w}, (u_s, v_s)\bigr)dB_s, \,\,  t \le s \le T,  \label{Eq2.3}
\end{align}
where $w_{\cdot} \triangleq (u_{\cdot}, v_{\cdot})$ is a pair of $(U, V)$-valued measurable decision processes. Furthermore, we also suppose that the data $(\xi^{Target}, L)$ take the following forms
\begin{align}
 \xi^{Target} = \Psi_l(X_T^{t,x;w}) \quad \text{and} \quad L_s = h(s, X_s^{t,x;w}).  \label{Eq2.4}
 \end{align}
Moreover, we introduce the following two risk-value functions w.r.t. the {\it leader} and {\it follower}, i.e.,
\begin{align}
 \text{\it leader:} &\quad  V_l^{u}\bigl(t, x\bigr) = \rho_{t, T}^{g_l} \bigl[\xi_{t,T}^l\bigl(u, v\bigr)\bigr],  \,\, \text{such that, for a given} \,\, \hat{u}_{\cdot} \in \mathcal{U}_{[t,T]},\notag \\
   &  \quad  \exists \hat{v}_{\cdot} \in \Bigl\{\tilde{v}_{\cdot} \in \mathcal{V}_{[t,T]} \,\Bigl\vert\, \rho_{t, T}^{g_f} \bigl[\xi_{t,T}^f\bigl(\hat{u}, \tilde{v}\bigr)\bigl] \le \rho_{t, T}^{g_f} \bigl[\xi_{t,T}^f\bigl(\hat{u}, v\bigr)\bigl], \,\, \forall v_{\cdot} \in \mathcal{V}_{[t,T]},  \notag \\
   & \hspace{0.35in} \quad \quad \quad\quad \quad  \rho_{t, T}^{g_l} \bigl[\xi_{t,T}^l\bigl(\hat{u}, \tilde{v}\bigr)\bigr] \ge L_t, \quad \mathbb{P}-a.s. \Bigr\}, \label{Eq2.5}
\end{align}
where
\begin{align}
\xi_{t,T}^l\bigl(v, w\bigr) = \int_t^T c_l\bigl(s, X_s^{t,x;w}, v_s\bigr) ds + \Psi_l(X_T^{t,x;w}) \label{Eq2.6}
\end{align}
and similarly
\begin{align}
\text{\it follower:} \quad V_f^{v}\bigl(t, x\bigr) = \rho_{t, T}^{g_f} \bigl[\xi_{t,T}^f\bigl(u,v\bigr)\bigr],  \hspace{2in}   \label {Eq2.7}
\end{align}
where
\begin{align}
\xi_{t,T}^f \bigl(v,w\bigr) = \int_t^T c_f\bigl(s, X_s^{t,x;w}, w_s\bigr) ds + \Psi_f(X_T^{t,x;w}). \label{Eq2.8}
\end{align}
Taking into account Assumption~\ref{AS2} (and with Markovian risk-averse decisions), we can express the above two risk-value functions using reflected- and standard-BSDE as follows
\begin{align}
 V_l^{u}\bigl(t, x\bigr) &\triangleq \hat{Y}_s^{t,x;w} \notag \\
                                   &= \Psi_l(X_T^{t,x;w}) + \int_t^T g_l\bigl(s, X_s^{t,x;w}, \hat{Y}_s^{t,x;w}, \hat{Z}_s^{t,x;w} \bigr) ds \notag \\
                                                                      & \quad\quad + A_T^{t,x;w} - A_s^{t,x;w} - \int_t^T \hat{Z}_s^{t,x;w}dB_s, \label{Eq2.9}
\end{align}
where $\bigl\{ A_s^{t,x;w}\bigr\}$ is increasing and continuous, and
\begin{align*}
\int_t^T \bigl(\hat{Y}_s^{t,x;w} - h(s, X_s^{t,x;w}) \bigr) d A_s^{t,x;w} = 0,
\end{align*}
with $L_s = h(s, X_s^{t,x;w})$, and 
\begin{align*}
g_l\bigl(t, X_s^{t,x;w}, \hat{Y}_s^{t,x;w}, \hat{Z}_s^{t,x;w}\bigr) = c_l\bigl(s, X_s^{t,x;w}, u_s\bigr) + g\bigl(s, \hat{Y}_s^{t,x;w} \hat{Z}_s^{t,x;w}\bigr) 
\end{align*}
and
\begin{align}
 V_f^{v}\bigl(t, x\bigr) &\triangleq \tilde{Y}_s^{t,x;w} \notag \\
                                 &= \Psi_f(X_T^{t,x;w}) + \int_t^T g_f\bigl(s, X_s^{t,x;w} \tilde{Y}_s^{t,x;w}, \tilde{Z}_s^{t,x;w}\bigr) \Bigr\}ds \notag \\
                                                                      & \quad \quad \quad \quad - \int_t^T \tilde{Z}_s^{t,x;w}dB_s, \label{Eq2.10} 
\end{align}
where 
\begin{align*}
g_f\bigl(t, X_s^{t,x;w}, \tilde{Y}_s^{t,x;w}, \tilde{Z}_s^{t,x;w}\bigr) = c_f\bigl(s, X_s^{t,x;w}, v_s\bigr) + g\bigl(s, \tilde{Y}_s^{t,x;w}, \tilde{Z}_s^{t,x;w}\bigr). 
\end{align*}

Noting the conditions in \eqref{Eq1.7} and \eqref{Eq1.8} (see also Remark~\ref{R2.1}), then $\bigl(\hat{Y}_s^{t,x;w}, \hat{Z}_s^{t,x;w}, A_s^{t,x;w}\bigr)_{t \le s \le T}$ \,and \,$\bigl(\tilde{Y}_s^{t,x;w}, \tilde{Z}_s^{t,x;w}\bigr)_{t \le s \le T}$ are adapted solutions on $[t, T] \times \Omega$ and belong to $\mathcal{S}^2\bigl(t, T; \mathbb{R} \bigr) \times  \mathcal{H}^2\bigl(t, T; \mathbb{R}^{d} \bigr) \times \mathcal{S}^2\bigl(t, T; \mathbb{R}\bigr)$.
 
\begin{remark} \label{R2.1}
Here, it is worth remarking that, for a given continuous progressively real-valued process $\bigl\{L_t, 0 \le t \le T \bigr\}$ satisfying \eqref{Eq2.2} and for each $(t, x) \in [0, T] \times \mathbb{R}^d$, if there exists a unique triple $\bigl(\hat{Y}^{t,x;w}, \hat{Z}^{t,x;w}, A^{t,x;w}\bigr)$ of $\bigl\{\mathcal{F}_s^t\bigr\}$ progressively measurable processes, which solves the reflected-BSDE in \eqref{Eq2.9}. Then, this is equivalent to solve 
\begin{align*}
\bigl(\bar{Y}^{t,x;w}, \bar{Z}^{t,x;w}\bigr) = \bigl(\hat{Y}^{t,x;w} + A^{t,x;w}, \hat{Z}^{t,x;w}\bigr) \in \mathcal{S}^2\bigl(t, T; \mathbb{R} \bigr) \times  \mathcal{H}^2\bigl(t, T; \mathbb{R}^{d} \bigr) 
\end{align*}
of the following standard-BSDE
\begin{align*}
  \bar{Y}_s^{t,x;w} &= \Psi_l(X_T^{t,x;w}) + \int_t^T g_l\bigl(s, X_s^{t,x;w}, \bar{Y}_s^{t,x;w}, \bar{Z}_s^{t,x;w} \bigr) ds - \int_t^T \bar{Z}_s^{t,x;w}dB_s.
\end{align*}
\end{remark}

In what follows, we introduce a hierarchical framework that requires a certain level of risk-averseness be achieved for the {\it leader} as a priority over that of the {\it follower}. For example, suppose that the risk-averse decision for the {\it leader} $\hat{u}_{\cdot} \in \mathcal{U}_{[t,T]}$ is given. Then, the problem of finding an optimal risk-averse decision for the {\it follower}, i.e., $\hat{v}_{\cdot} \in \mathcal{V}_{[t,T]}$, which minimizes the accumulated risk-cost under $v$ is then reduced to finding an optimal risk-averse solution for
\begin{align}
\inf_{v_{\cdot} \in \mathcal{V}_{[t,T]}}  J_f\bigr[\bigl(\hat{u}, v\bigr)\bigl], \label{Eq2.11}
\end{align}
where 
\begin{align}
J_f\bigr[\bigl(\hat{u}, v\bigr)\bigl] = \rho_{t, T}^{g_f} \bigl[\xi_{t,T}^f\bigl(\hat{u},v\bigr)\bigr]. \label{Eq2.12}
\end{align}
Note that, for a given $\hat{u}_{\cdot} \in \mathcal{U}_{[t,T]} $, if the forward-backward stochastic differential equations (FBSDEs) in \eqref{Eq2.3}, \eqref{Eq2.10} and the reflected-BSDE in \eqref{Eq2.9} admit unique solutions, then we have
\begin{align}
  \hat{v}_{\cdot} \triangleq F(\hat{u}_{\cdot})  \in \Bigl\{\tilde{v}_{\cdot} \in \mathcal{V}_{[t,T]} \,\Bigl\vert\, \rho_{t, T}^{g_f} \bigl[\xi_{t,T}^f\bigl(\hat{u}, \tilde{v}\bigr)\bigl] \le \rho_{t, T}^{g_f} \bigl[\xi_{t,T}^f\bigl(\hat{u}, v\bigr)\bigl], \,\, \forall v_{\cdot} \in \mathcal{V}_{[t,T]},  \notag \\
   \quad  \rho_{t, T}^{g_l} \bigl[\xi_{t,T}^l\bigl(\hat{u}, \tilde{v}\bigr)\bigr] \ge L_t, \quad \mathbb{P}-a.s. \Bigr\}, \label{Eq2.13}
\end{align}
for some measurable mapping $F \colon \hat{u}_{\cdot} \in \mathcal{U}_{[t,T]} \rightsquigarrow \hat{v}_{\cdot} \in \mathcal{V}_{[t,T]}$. Moreover, if we substitute $\hat{w} = (\hat{u},F(\hat{u}))$ into \eqref{Eq2.3}, then the corresponding solution $X_s^{t,x;\hat{w}}$ depends uniformly on $\hat{u}_{\cdot}$ for $s \in [t, T]$. Further, the risk-averse decision problem (which minimizes the accumulated risk-cost under $u$ w.r..t the {\it leader}) is then reduced to finding an optimal risk-averse solution for
\begin{align}
\inf_{u_{\cdot} \in \mathcal{U}_{[t,T]}}  J_l\bigl[\bigl(u, F(u)\bigr) \bigr], \label{Eq2.14}
\end{align}
where
\begin{align}
J_l\bigr[\bigl(u, F(u)\bigr)\bigl] = \rho_{t, T}^{g_l} \bigl[\xi_{t,T}^l\bigl(u,F(u)\bigr)\bigr]. \label{Eq2.15}
\end{align}

\begin{remark} \label{R2.2}
Note that the generator functionals $g_l$ and $g_f$ contain a common term $g$ that acts on different processes (see also equation~\eqref{Eq2.9} and \eqref{Eq2.10}). Moreover, due to differing cost functionals w.r.t. the {\it leader} and {\it follower}, $\rho_{t, T}^{g_l} \bigl[\,\cdot\,]$ and $\rho_{t, T}^{g_f} \bigl[\,\cdot\,]$ provide multi-structure dynamic risk measures.
\end{remark}

Next, we introduce the definition of admissible hierarchical risk-averse decision system $\Sigma_{[t, T]}$, with time-consistent, dynamic risk measures, which provides a logical construct for our main results (e.g., see also \cite{LIW14}).
\begin{definition} \label{Df2}
For a given finite-time horizon $T>0$, we call $\Sigma_{[t, T]}$ an admissible hierarchical risk-averse decision system, if it satisfies the following conditions:
\begin{itemize}
\item $\bigl(\Omega, \mathcal{F},\{\mathcal{F}_t \}_{t \ge 0}, \mathbb{P}\bigr)$ is a complete probability space;
\item $\bigl\{B_s\bigr\}_{s \ge t}$ is a $d$-dimensional standard Brownian motion defined on $\bigl(\Omega, \mathcal{F}, \mathbb{P}\bigr)$ over $[t, T]$ and $\mathcal{F}^t \triangleq \bigl\{\mathcal{F}_s^t\bigr\}_{s \in [t, T]}$, where $\mathcal{F}_s^t = \sigma\bigl\{\bigl(B_s; \,t \le s \le T \bigr)\bigr\}$ is augmented by all $\mathbb{P}$-null sets in $\mathcal{F}$;
\item $u_{\cdot} \colon \Omega \times [s, T]  \rightarrow U$ and $v_{\cdot} \colon \Omega \times [s, T]  \rightarrow V$ are $\bigl\{\mathcal{F}_s^t\bigr\}_{s \ge t}$-adapted processes on $\bigl(\Omega, \mathcal{F}, \mathbb{P}\bigr)$ with 
\begin{align*}
\mathbb{E} \int_{s}^{T} \vert u_{\tau}\vert^2 d \tau < \infty \quad \text{and}  \quad \mathbb{E} \int_{s}^{T} \vert v_{\tau}\vert^2 d \tau < \infty, \quad s \in [t, T];
\end{align*}
\item There exists at least one measurable mapping $F \colon u_{\cdot} \in \mathcal{U}_{[t,T]} \rightsquigarrow v_{\cdot} \in \mathcal{V}_{[t,T]}$ with $v_{\cdot} = F\bigl(u_{\cdot} \bigr)$ whenever $u_{\cdot} \in \mathcal{U}_{[t,T]}$ satisfies \eqref{Eq2.12};
\item For any $x \in \mathbb{R}^d$, the FBSDEs in \eqref{Eq2.3}, \eqref{Eq2.10} and the reflected-BSDE in \eqref{Eq2.9} admit a unique solution set $\bigl\{X_{\cdot}^{s,x;u}, (\hat{Y}_{\cdot}^{s,x;w}, \hat{Z}_{\cdot}^{s,x;w},  A_{\cdot}^{s,x;w}), (\tilde{Y}_{\cdot}^{s,x;w}, \tilde{Z}_{\cdot}^{s,x;w})\bigr\}$ on $\bigl(\Omega, \mathcal{F}, \mathcal{F}^t, \mathbb{P}\bigr)$ with \,$w_{\cdot}=\bigl(u_{\cdot}, F(u_{\cdot})\bigr)$.
\end{itemize}
\end{definition}

Then, with restriction to the above admissible hierarchical risk-averse decision system, we can state the risk-averse decision problem as follow.

{\bf Problem~(P)}. Find a pair of risk-averse strategies $(u_{\cdot}^{\ast}, v_{\cdot}^{\ast}) \in \mathcal{U}_{[0,T]} \otimes \mathcal{V}_{[0,T]}$ w.r.t. the {\it leader} and that of the {\it follower}, with $\xi^{Target} \in L^2(\Omega, \mathcal{F}_T, \mathbb{P}; \mathbb{R})$, such that
\begin{align}
 u_{\cdot}^{\ast} \in \Bigl\{ \arginf J_l\bigr[\bigl(u, v\bigr)\bigl] \Bigl \vert v_{\cdot} = F(u_{\cdot}) \,\, \& \,\, (u_{\cdot}, F(u_{\cdot})) \,\, \text{restricted to} \,\,\Sigma_{[0, T]} \Bigr\} \label{Eq2.16}
\end{align}
and
\begin{align}
  v_{\cdot}^{\ast} \in \Bigl\{ \arginf J_f\bigr[\bigl(u, v\bigr)\bigl] \Bigl \vert  v_{\cdot}^{\ast} = F(u_{\cdot}^{\ast}) \,\, \& \,\,  (u_{\cdot}^{\ast},F(u_{\cdot}^{\ast})) \, \text{restricted to} \,\Sigma_{[0, T]} \Bigr\}, \label{Eq2.17}
\end{align}
where $F$ is a measurable mapping the set $\mathcal{U}_{[0,T]}$ onto $\mathcal{V}_{[0,T]}$ and, furthermore, the accumulated risk-costs $J_l$ and $J_f$ over the time-interval $[0, T]$ are given  
\begin{align}
J_l\bigr[\bigl(v, w\bigr)\bigl] = \int_0^T c_l\bigl(s, X_s^{0,x;w}, u_s\bigr) ds + \Psi_l(X_T^{0,x;w}), \quad \Psi_l(X_T^{0,x;w}) = \xi^{Target}  \label{Eq2.18}
\end{align}
and
\begin{align}
J_f\bigr[\bigl(v, w\bigr)\bigl] = \int_0^T c_f\bigl(s, X_s^{0,x;w}, v_s\bigr) ds + \Psi_f(X_T^{0,x;w}),  \label{Eq2.19}
\end{align}
where $X_0^{0,x;w} = x$ and $w_{\cdot}=(u_{\cdot},v_{\cdot})$.\footnote{Note that
\begin{align*}
  v_{\cdot}^{\ast} \in \Bigl\{ \arginf J_f\bigr[\bigl(u, v\bigr)\bigl] \Bigl \vert  u_{\cdot} = F^{-1}(v_{\cdot}) \,\, \& \,\,  (F^{-1}(v_{\cdot}),v) \, \text{restricted to} \,\Sigma_{[0, T]} \Bigr\},
\end{align*}
 where $F^{-1} \colon v_{\cdot} \in \mathcal{V}_{[t,T]} \rightsquigarrow u_{\cdot} \in \mathcal{U}_{[t,T]}$.}

In the following section, we establish the existence of optimal risk-averse solutions, in the sense of viscosity, for the risk-averse optimization problems in \eqref{Eq2.16} and \eqref{Eq2.17} with restriction to $\Sigma_{[0, T]}$. Note that, for a given $u_{\cdot} \in \mathcal{U}_{[0,T]}$, the risk-averse optimization problem in \eqref{Eq2.17} has a unique solution on $\mathcal{V}_{[0,T]}$. Moreover, as we will see later on, the problem in \eqref{Eq2.16} makes sense if the {\it follower} is involved not only in minimizing his own accumulated risk-cost (in response to the risk-averse decision of the {\it leader}) but also in minimizing that of the {\it leader}.

\section{Main results} \label{S3}
In this section, we present our main results, where we introduce a  hierarchical framework under which the {\it follower} is required to respond optimally to the risk-averse decision of the {\it leader} so as to achieve an overall risk-averseness. Moreover, such a framework allows us to establish the existence of optimal risk-averse solutions, in the sense of viscosity solutions, to the associated risk-averse dynamic programming equations. 

We now state the following propositions that will be useful for proving our main results in Subsections~\ref{S3.1} and \ref{S3.3}.
\begin{proposition} \label{P1}
Suppose that the generator functional $g$ satisfies Assumption~\ref{AS1}. Further, let the statements in \eqref{Eq1.7}, \eqref{Eq1.8} and Assumption~\ref{AS2} along with \eqref{Eq2.4} hold true. Then, for any $(t,x) \in [0, T] \times \mathbb{R}^d$ and for every $w_{\cdot}=(u_{\cdot}, v_{\cdot}) \in \mathcal{U}_{[t,T]} \otimes \mathcal{V}_{[t,T]}$, the FBSDEs in \eqref{Eq2.3}, \eqref{Eq2.10} and the reflected-BSDE in \eqref{Eq2.9} admit unique adapted solutions
\begin{eqnarray}
\left.\begin{array}{c}
X_{\cdot}^{t,x;w} \in  \mathcal{S}^2\bigl(t, T; \mathbb{R}^{d} \bigr) ~~~~~~\\
\bigl(\hat{Y}_{\cdot}^{t,x;w}, \hat{Z}_{\cdot}^{t,x;w}, A_{\cdot}^{t,x;w}\bigr) \in \mathcal{S}^2\bigl(t, T; \mathbb{R} \bigr) \times  \mathcal{H}^2\bigl(t, T; \mathbb{R}^{d} \bigr) \times \mathcal{S}^2\bigl(t, T; \mathbb{R} \bigr) \\
\bigl(\tilde{Y}_{\cdot}^{t,x;w}, \tilde{Z}_{\cdot}^{t,x;w}\bigr) \in \mathcal{S}^2\bigl(t, T; \mathbb{R} \bigr) \times  \mathcal{H}^2\bigl(t, T; \mathbb{R}^{d} \bigr)
\end{array}\right\}  \label{Eq3.1}
\end{eqnarray}
Furthermore, the risk-values w.r.t. the {\it leader} and {\it follower}, i.e., $V_l^{u}\bigl(t, x\bigr)$ and $V_f^{v}\bigl(t, x\bigr)$, are deterministic.
\end{proposition}

\begin{proof}
Notice that $m$ and $\sigma$ are bounded and Lipschitz continuous w.r.t. $(t,x) \in [0, T] \times \mathbb{R}^d$ and uniformly for $(u, v) \in U \times V$. Then, for any $(t,x) \in [0, T] \times \mathbb{R}^d$ and $w_{\cdot}=(u_{\cdot}, v_{\cdot})$ are progressively measurable processes, there always exists a unique path-wise solution $X_{\cdot}^{t,x;w} \in  \mathcal{S}^2\bigl(t, T; \mathbb{R}^{d} \bigr)$ for the forward SDE in \eqref{Eq2.3}. On the other hand, consider the following BSDEs
\begin{align}
 -d \bar{Y}_s^{t,x;w} = g_l\bigl(s, X_s^{t,x;w}, \bar{Y}_s^{t,x;w}, \hat{Z}_s^{t,x;w}\bigr) ds - \hat{Z}_s^{t,x;w} dB_s, \label{EqP1.1}
\end{align}
where  
\begin{align*}
\bar{Y}_T^{t,x;w} = A_T^{t,x;w} + \int_t^T c_l\bigl(\tau, X_{\tau}^{t,x;w}, u_{\tau}\bigr) d\tau + \Psi_l(X_T^{t,x;w})
\end{align*}
and 
\begin{align}
 -d \breve{Y}_s^{t,x;u} =  g_f\bigl(s, X_s^{t,x;w}, \breve{Y}_s^{t,x;w}, \tilde{Z}_s^{t,x;w}\bigr) ds - \tilde{Z}_s^{t,x;w} dB_s, \label{EqP1.2}
\end{align}
where 
\begin{align*}
\breve{Y}_T^{2;t,x;u} = \int_t^T c_f\bigl(\tau, X_{\tau}^{t,x;w}, v_{\tau}\bigr) d\tau + \Psi_f(X_T^{t,x;w}). 
\end{align*}
From Lemma~\ref{L1} (and see also Remark~\ref{R2.1}), the equations in \eqref{EqP1.1} and \eqref{EqP1.2} admit unique solutions $\bigl(\bar{Y}_{\cdot}^{t,x;w}, \hat{Z}_{\cdot}^{t,x;w}\bigr)$ and $\bigl(\breve{Y}_{\cdot}^{t,x;w}, \tilde{Z}_{\cdot}^{t,x;w}\bigr)$ in $\mathcal{S}^2\bigl(t, T; \mathbb{R} \bigr) \times  \mathcal{H}^2\bigl(t, T; \mathbb{R}^{d} \bigr)$. Furthermore, if we introduce the following
\begin{align*}
\hat{Y}_s^{t,x;w} = \bar{Y}_s^{t,x;w} - A_s^{t,x;w} - \int_t^s c_l\bigl(\tau, X_{\tau}^{t,x;w}, u_{\tau}\bigr) d\tau,   \quad s \in [t, T]
\end{align*}
and
\begin{align*}
\tilde{Y}_s^{t,x;w} = \breve{Y}_s^{t,x;w} - \int_t^s c_f\bigl(\tau, X_{\tau}^{t,x;w}, v_{\tau}\bigr) d\tau,  \quad s \in [t, T].
\end{align*}
Then, the forward of the reflected BSDE in \eqref{Eq2.9} and that of the BSDE in \eqref{Eq2.10} hold, respectively, with $\bigl(\hat{Y}_{\cdot}^{t,x;w}, \hat{Z}_{\cdot}^{t,x;w}, A_{\cdot}^{t,x;w}\bigr)$ and $\bigl(\tilde{Y}_{\cdot}^{t,x;w}, \tilde{Z}_{\cdot}^{t,x;w}\bigr)$. Moreover, we also observe that $\hat{Y}_t^{t,x;w}$ and $\tilde{Y}_t^{t,x;w}$ are deterministic. This completes the proof of Proposition~\ref{P1}. \qed
\end{proof}

\begin{proposition} \label{P2}
Let $(t,x) \in [0, T] \times \mathbb{R}^d$ and $w_{\cdot} = (u_{\cdot}, v_{\cdot}) \in \mathcal{U}_{[t,T]} \otimes \mathcal{V}_{[t,T]}$ be restricted to $\Sigma_{[t, T]}$ (cf. Definition~\ref{Df2}). Then, for any $r \in [t, T]$ and $\mathbb{R}^d$-valued $\mathcal{F}_r^t$-measurable random variable $\eta$, we have
\begin{align}
 V_l^{u}\bigl(r, \eta\bigr) &= \hat{Y}_r^{t,x;w} \notag\\
                                    &\triangleq \rho_{r, T}^{g_l} \Bigl[\int_r^T c_l\bigl(s, X_s^{r,\eta;w}, u_s\bigr) ds + \Psi_l(X_T^{r,\eta;w}) \Bigr], \quad \mathbb{P}{\text-a.s.}  \label{Eq3.2}
\end{align}
and
\begin{align}
 V_f^{v}\bigl(r, \eta\bigr) &= \tilde{Y}_r^{t,x;w} \notag\\
                                    &\triangleq \rho_{r, T}^{g_f} \Bigl[\int_r^T c_f\bigl(s, X_s^{r,\eta;w}, v_s\bigr) ds + \Psi_f(X_T^{r,\eta;w}) \Bigr], \quad \mathbb{P}{\text-a.s.} \label{Eq3.3}
\end{align}
\end{proposition}

\begin{proof}
For any $r \in [t, T]$, with $t \in [0, T]$, we consider the following probability space $\bigl(\Omega, \mathcal{F}, \mathbb{P}\bigl(\cdot\vert \mathcal{F}_r^t\bigr), \{\mathcal{F}^t\}\bigr)$ and notice that $\eta$ is deterministic under this probability space. Then, for any $s \ge r$, there exist progressively measurable processes $\psi_1$ and $\psi_2$ such that
\begin{align}
 \bigl(u_s(\Omega), v_s(\Omega)\bigl) &= \bigl(\psi_1(\Omega, B_{\cdot \wedge s}(\Omega)), \psi_2(\Omega, B_{\cdot \wedge s}(\Omega)) \bigr), \\  
                                                               &= \bigl(\psi_1(s, \bar{B}_{\cdot \wedge s}(\Omega) + B_{r}(\Omega)), \psi_2(s, \bar{B}_{\cdot \wedge s}(\Omega) + B_{r}(\Omega)) \bigr), \label{EqP2.1}
\end{align}
where $\bar{B}_s = B_s - B_r$ is a standard $d$-dimensional brownian motion. Note that the pairs $\bigl(u_{\cdot}, v_{\cdot}\bigl)$ are $\mathcal{F}_r^t$-adapted processes, then we have the following restriction w.r.t. $\Sigma_{[t, T]}$
\begin{align}
 \bigl(\Omega, \mathcal{F}, \{\mathcal{F}^t\}, \mathbb{P}\bigl(\cdot \vert \mathcal{F}_r^t\bigr)(\omega'), B_{\cdot}, \bigl(u_{\cdot}, v_{\cdot} \bigl)\bigr) \in \Sigma_{[t, T]}, \label{EqP2.2}
\end{align}
where $\omega' \in \Omega'$ such that $\Omega' \in \mathcal{F}$, with $\mathbb{P}(\Omega')=1$. Furthermore, noting Lemma~\ref{L1}, if we work under the probability space $\bigl(\Omega', \mathcal{F}, \mathbb{P}\bigl(\cdot\vert \mathcal{F}_r^t\bigr)\bigr)$, then both statements in \eqref{Eq3.2} and \eqref{Eq3.3} hold $\mathbb{P}$-{\it almost surely}. This completes the proof of Proposition~\ref{P2}. \qed
\end{proof}

In what follows, we restrict our discussion w.r.t. the generator functional $g_f$ which is associated with the {\it follower}. Moreover, for $w=(u, v) \in U \times V$ and any $ \phi(x) \in C_0^{\infty}(\mathbb{R}^d)$, we introduce a family of second-order linear operators,  associated with \eqref{Eq1.1}, as follow
\begin{align}
 \mathcal{L}_{t}^{(u,v)} \phi(x)  = \dfrac{1}{2} \operatorname{tr} \Bigl\{a(t, x, w) D_{x}^2 \phi(x)\Bigr\} + m(t, x, w) D_{x} \phi(x), \quad t \in [0, T],  \label{Eq3.4}
\end{align}
where $a(t, x, w) = \sigma(t, x, w) \sigma^T(t, x, w)$, $D_{x}$ and $D_{x}^2$, (with $D_{x}^2 = \bigl({\partial^2 }/{\partial x_i \partial x_j} \bigr)$) are the gradient and the Hessian (w.r.t. the variable $x$), respectively. Further, on the space $C_b^{1,2}([t, T] \times \mathbb{R}^d)$, for any $(t,x) \in [0, T] \times \mathbb{R}^d$, we consider the following Hamilton-Jacobi-Bellman (HJB) partial differential equation (PDE)
\begin{eqnarray}
\left.\begin{array}{r}
 \dfrac{\partial \varphi(t, x)}{\partial t}  + \inf_{v \in V} \Bigl\{\mathcal{L}_{t}^{(u,v)} \varphi(t, x) ~~~ \hspace{1.5in} \quad \quad \quad  \\
  + g_f\bigl(t, \varphi(t, x), D_x \varphi(t, x) \cdot \sigma(t, x, (u,v))\bigr)\Bigr\} = 0  \\
   \text{where} \,\, u \,\, \text{is assumed to be given} 
\end{array}\right\}  \label{Eq3.6}
\end{eqnarray}
with the following boundary condition
\begin{align}
\varphi(T, x) = \Psi_f(T, x), \quad x \in \mathbb{R}^d. \label{Eq3.7}
\end{align}

\begin{remark} \label{R3.1}
Here, we remark that the above equation in \eqref{Eq3.6} together with \eqref{Eq3.7}, is associated with the risk-averse decision problem for the {\it follower}, restricted to $\Sigma_{[t,T]}$ (cf. Definition~\ref{Df2}), with cost functional in \eqref{Eq2.19}. Moreover, it represents a generalized HJB equation with additional terms $g_f$. Note that the problem of FBSDEs and reflected BSDEs (cf. equations~\eqref{Eq2.3}, \eqref{Eq2.9} and \eqref{Eq2.10}) and the solvability of the related HJB partial differential equations (PDEs) have been well studied in literature (e.g., see \cite{Ant93}, \cite{ElKKPQ97}, \cite{HUP95}, \cite{LIW14}, \cite{MaPY94}, \cite{ParT99}, \cite{Pen91} and \cite{Pen92}).
\end{remark}

Next, we recall the definition of viscosity solutions for \eqref{Eq3.6} together with \eqref{Eq3.7} (e.g., see \cite{CraIL92}, \cite{FleS06} or \cite{Kry08} for additional discussions on the notion of viscosity solutions).

\begin{definition} \label{Df3}
The functions $\varphi \colon [0, T] \times \mathbb{R}^d$ is viscosity solutions for \eqref{Eq3.6} together with the boundary conditions in \eqref{Eq3.7}, if the following conditions hold
\begin{enumerate} [(i)]
\item for every $\psi \in C_b^{1,2}([0, T], \times \mathbb{R}^d)$ such that $\psi \ge \varphi$ on $[0, T] \times \mathbb{R}^d$,
\begin{align}
\sup_{(t,x)} \bigl\{\varphi(t,x) - \psi(t,x) \bigr\} = 0, \label{Eq3.8}
\end{align}
and for $(t_{0},x_{0}) \in  [0, T] \times \mathbb{R}^d$ such that  $\psi(t_{0}, x_{0})=\varphi(t_{0}, x_{0})$ (i.e., a local maximum at $(t_{0},x_{0})$), then we have
\begin{align}
 \dfrac{\partial \psi(t_{0},x_{0})}{\partial t}  &+ \inf_{v \in V} \Bigl\{\mathcal{L}_{t}^{(u,v)} \psi(t_{0},x_{0}) \notag \\
 & + g_f\bigl(t_{0}, x_{0}, \psi(t_{0},x_{0}), D_x \psi(t_{0},x_{0}) \cdot \sigma(t_{0},x_{0}, (u,v))\bigr)\Bigr\} \ge 0 \label{Eq3.10}
\end{align}
\item for every $\psi \in C_b^{1,2}([0, T], \times \mathbb{R}^d)$ such that $\psi \le \varphi$ on $[0, T] \times \mathbb{R}^d$,
\begin{align}
\inf_{(t,x)} \bigl\{ \varphi(t,x) - \psi(t,x) \bigr\} = 0,  \label{Eq3.11}
\end{align}
and for $(t_{0},x_{0}) \in  [0, T] \times \mathbb{R}^d$ such that  $\psi(t_{0}, x_{0})=\varphi(t_{0}, x_{0})$ (i.e., a local minimum at $(t_{0},x_{0})$), then we have
\begin{align}
 \dfrac{\partial \psi(t_{0},x_{0})}{\partial t}  &+ \inf_{v \in V} \Bigl\{ \mathcal{L}_{t}^{(u,v)} \psi(t_{0},x_{0}) \notag \\
 & + g_f\bigl(t_{0},x_{0}, \psi(t_{0},x_{0}), D_x \psi(t_{0},x_{0}) \cdot \sigma(t_{0},x_{0}, (u,v))\bigr)\Bigr\} \le 0. \label{Eq3.13}
\end{align}
\end{enumerate}
\end{definition}

\subsection{On the risk-averse optimality condition for the {\it follower}} \label{S3.1}
Suppose that, for a given {\it leader's} risk-averse decision $\hat{u}_{\cdot} \in \mathcal{U}_{[t,T]}$, the decision for the {\it follower} is an optimal solution to \eqref{Eq2.7}. Then, with restriction to $\Sigma_{[t, T]}$, such a solution is characterized by the following propositions (i.e., Propositions~\ref{P3}, \ref{P4} and \ref{P5}). 
\begin{proposition} \label{P3}
Suppose that the generator functional $g$ satisfies Assumption~\ref{AS1}. Further, let the statements in \eqref{Eq1.7}, \eqref{Eq1.8} and Assumption~\ref{AS2} along with \eqref{Eq2.4} hold true. Let $\hat{u}_{\cdot} \in \mathcal{U}_{[t,T]}$ be given, then the risk-value function w.r.t. the {\it follower} is given by
\begin{align}
 V_f^{v}\bigl(t, x\bigr) = \inf_{v_{\cdot} \in \mathcal{V}_{[t,r]} \bigl \vert \Sigma_{[t, T]}}  \rho_{t, r}^{g_f} \Bigl[\int_t^r c_f\bigl(s, X_s^{t,x;w}, v_s\bigr) ds +V_f^{v}\bigl(r, X_r^{t,x;w}\bigr) \Bigr] \label{Eq3.14}
\end{align}
 for any $(t,x) \in [0, T] \times \mathbb{R}^d$ and $r \in [t, T]$, with $w=(\hat{u}, v)$.
\end{proposition}
\begin{proof}
Notice that $\hat{u}_{\cdot} \in \mathcal{U}_{[t,T]}$ is given. Then, for any $\epsilon > 0$, there exists $\tilde{v}_{\cdot} \in \mathcal{V}_{[t,T]}$ such that $V_f^{v}\bigl(t, x\bigr)  + \epsilon \ge V_f^{\tilde{v}}\bigl(t, x\bigr)$. Further, if we applying the properties of time-consistency and translation to $V_f^{\tilde{v}}\bigl(t, x\bigr)$, then we have
\begin{align}
&V_f^{v}\bigl(t, x\bigr)  + \epsilon \ge V_f^{\tilde{v}}\bigl(t, x\bigr) \notag \\
                                          &\quad = \rho_{t, r}^{g_f} \Bigl[\rho_{r, T}^{g_f} \Bigl[\int_t^T c_f\bigl(s, X_s^{t,x;\tilde{w}}, \tilde{v}_s\bigr) ds + \Psi_f(X_T^{t,x;\tilde{w}}) \Bigr] \Bigr] \notag \\
                                          &\quad = \rho_{t, r}^{g_f} \Bigl[ \int_t^r c_f\bigl(s, X_s^{t,x;\tilde{w}}, \tilde{v}_s\bigr) ds + \rho_{r, T}^{g_f} \Bigl[ \int_r^T c_f\bigl(s, X_s^{t,x;\tilde{w}}, \tilde{v}_s\bigr) ds + \Psi_f(X_T^{t,x;\tilde{w}}) \Bigr] \Bigr], 
                                                    \label{EqP3.1}
\end{align}
where $\tilde{w}_{\cdot}=(\hat{u}_{\cdot}, \tilde{v}_{\cdot})$ is restricted to $\Sigma_{[t, T]}$. Moreover, if we apply Proposition~\ref{P2}, then we have  
\begin{align}
V_f^{v}\bigl(t, x\bigr) + \epsilon &\ge \rho_{t, r}^{g_f} \Bigl[ \int_t^r c_f\bigl(s, X_s^{t,x;\tilde{w}}, \tilde{v}_s\bigr) ds + V_f^{\tilde{v}}\bigl(r, X_r^{t,x;\tilde{w}}\bigr)\Bigr] \notag \\
                                                   &\ge \rho_{t, r}^{g_f} \Bigl[ \int_t^r c_f\bigl(s, X_s^{t,x;\tilde{w}}, \tilde{v}_s\bigr) ds + V_f^{v}\bigl(r, X_r^{t,x;\tilde{w}}\bigr)\Bigr] \notag \\
                                                   &\ge \inf_{v_{\cdot} \in \mathcal{V}_{[t,r]} \bigl \vert \Sigma_{[t, T]}} \rho_{t, r}^{g_f} \Bigl[ \int_t^r c_f\bigl(s, X_s^{t,x;\tilde{w}}, \tilde{v}_s\bigr) ds + V_f^{v}\bigl(r, X_r^{t,x;w}\bigr)\Bigr]. \label{EqP3.2}
\end{align}
Since $\epsilon$ is arbitrary, we obtain \eqref{Eq3.14}. On the other hand, to show the reverse inequality $``\le"$, let $\tilde{v}_{\cdot}$ (which is restricted to $\Sigma_{[t, T]}$) be an $\epsilon$-optimal solution, for a fixed $\epsilon >0$, to the the problem on the right-hand side of \eqref{Eq3.14}.That is,
\begin{align}
\inf_{v_{\cdot} \in \mathcal{V}_{[t,r]} \bigl \vert \Sigma_{[t, T]}} \rho_{t, r}^{g_f} \Bigl[ \int_t^r &c_f\bigl(s, X_s^{t,x;\tilde{w}}, \tilde{v}_s\bigr) ds + V_f^{v}\bigl(r, X_r^{t,x;w}\bigr)\Bigr] + \epsilon \notag \\                                                   & \quad \ge \rho_{t, r}^{g_f} \Bigl[ \int_t^r c_f\bigl(s, X_s^{t,x;\tilde{w}}, \tilde{v}_s\bigr) ds + V_f^{v}\bigl(r, X_r^{t,x;\tilde{w}}\bigr)\Bigr].  \label{EqP3.3}
\end{align}
Then, for every $y \in \mathbb{R}^d$, let $\tilde{v}_{\cdot}(y) \in \mathcal{V}_{[t,T]}$ be such that $V_f^{v}\bigl(r, y\bigr) + \epsilon \ge V_f^{\tilde{v}(y)}\bigl(t, x\bigr)$ and restricted to $\Sigma_{[t, T]}$. Due to the measurable selection theorem, we may assume that the function $y \rightarrow \tilde{v}(y)$ is Borel measurable. Further, suppose that a control function $v_{\cdot}^0$ is defined as follow
\begin{eqnarray}
v_s^0 =\left\{\begin{array}{l l}
 \bar{v}_s,  & s \in [t, r)\\
  \tilde{v}_s(X_s^{t,x;\bar{w}}), & s \in [r, T].
\end{array}\right.  \label{EqP3.4}
\end{eqnarray}
Note that, from the above definition, $v_{\cdot}^0$ is restricted to $\Sigma_{[t, T]}$. Then, using the properties of the monotonicity, translation and time-consistency, we obtain the following
\begin{align}
 \rho_{t, r}^{g_f} &\Bigl[ \int_t^r c_f\bigl(s, X_s^{t,x;\bar{w}}, \bar{v}_s\bigr) ds + V_f^{\bar{w}}\bigl(r, X_r^{t,x;\bar{w}}\bigr)\Bigr] \notag \\
                                                   &\ge \rho_{t, r}^{g_f} \Bigl[ \int_t^r c_f\bigl(s, X_s^{t,x;\bar{w}}, \bar{v}_s\bigr) ds + V_f^{\tilde{v}_s(X_s^{t,x;\bar{w}})}\bigl(r, X_r^{t,x;\bar{w}}\bigr) -\epsilon \Bigr], \,\, \text{with} \,\, \bar{w} =(\hat{u}, \bar{v}) \notag \\
                                                   &\ge \rho_{t, T}^{g_f} \Bigl[ \int_t^T c_f\bigl(s, X_s^{t,x;w^0}, \bar{v}_s^0\bigr) ds + \Psi_f\bigl(X_T^{t,x;w^0}\bigr)\Bigr] -\epsilon, \,\, \text{with} \,\, w^0 =(\hat{u}, v^0)  \notag \\
                                                   &= V_f^{v^{0}} \bigl(t, x\bigr) - \epsilon. \label{EqP3.5}
\end{align}
If we combine the inequalities from \eqref{EqP3.3} and \eqref{EqP3.5}, then we have
\begin{align}
\inf_{v_{\cdot} \in \mathcal{V}_{[t,r]} \bigl \vert \Sigma_{[t, T]}} \rho_{t, r}^{g_f} \Bigl[ \int_t^r c_f\bigl(s, X_s^{t,x;w}, v_s\bigr) ds + V_f^{v}\bigl(r, X_r^{t,x;w}\bigr)\Bigr] &+ \epsilon                                 \ge V_f^{v^0}\bigl(t, x\bigr) - \epsilon \notag \\
                                 &\ge V_f^{v}\bigl(t, x\bigr) - \epsilon.  \label{EqP3.6}
\end{align}
Note that, since $\epsilon$ is arbitrary, we obtain \eqref{Eq3.14}. This completes the proof of Proposition~\ref{P3}. \qed
\end{proof}

Then, we have the following results (i.e., Propositions~\ref{P4} and \ref{P5}) that characterize the measurable mapping $F$ in \eqref{Eq2.7}.
\begin{proposition}\label{P4}
Suppose that the generator functional $g$ satisfies Assumption~\ref{AS1}. Let $V$ be a compact set in $\mathbb{R}^{d}$ and $\hat{u}_{\cdot} \in \mathcal{U}_{[t,T]}$ be given. Then, the risk-value function $V_f^{v}\bigl(\cdot, \cdot\bigr)$ is the viscosity solution of \eqref{Eq3.6} with boundary condition $\Psi_f(T, x)$ for $x \in \mathbb{R}^d$ and with $w=(\hat{u}, v)$.
\end{proposition}

\begin{proof}
Suppose that $\varphi \in C_b^{1,2}([0, T] \times \mathbb{R}^d)$ and assume that $\varphi \ge V_f^{v}$ on $[0, T] \times \mathbb{R}^d$ and $\max_{(t,x)} \bigl[V_f^{v}(t,x) - \varphi(t,x)\bigr] = 0$. We consider a point $(t_{0},x_{0}) \in [0, T] \times \mathbb{R}^d$ so that $\varphi(t_{0},x_{0})= V_f^{v}(t_{0},x_{0})$ (i.e., a local maximum at $(t_{0},x_{0})$). Further, for a small $\delta t > 0$, we consider a constant control $v_s=\alpha$ for $s \in [t_{0},t_{0} + \delta t]$. Then, from \eqref{Eq3.14}, we have
\begin{align}
 &\varphi(t_{0},x_{0}) = V_f^{v}(t_{0},x_{0}) \notag \\
                                        &\quad \le \rho_{t_{0},t_{0} + \delta t}^{g_f} \Bigl[\int_{t_{0}}^{t_{0} + \delta t} c_f\bigl(s, X_s^{t_{0},x_{0};w}, \alpha \bigr) ds + V_f^{v}(t_{0} + \delta t, X_{t_{0} + \delta t}^{t_{0},x_{0};w})\Bigr]  \notag \\
                                       &\quad \le \rho_{t_{0},t_{0} + \delta t}^{g_f} \Bigl[\int_{t_{0}}^{t_{0} + \delta t} c_f\bigl(s, X_s^{t_{0},x_{0};w}, \alpha \bigr) ds + \varphi(t_{0} + \delta t, X_{t_{0} + \delta t}^{t_{0},x_{0};w})\Bigr],  \,\, \text{with} \,\, w=(\hat{u}, \alpha). \label{EqP4.1}
\end{align}
Using the translation property of $\rho_{t_{0},t_{0} + \delta t}^{g_f}[\,\cdot\,]$, we obtain the following inequality
\begin{align}
\rho_{t_{0},t_{0} + \delta t}^{g_f} \Bigl[\int_{t_{0}}^{t_{0} + \delta t} c_f\bigl(s, X_s^{t_{0},x_{0};w}, \alpha \bigr) ds + \varphi(t_{0} + \delta t, X_{t_{0} + \delta t}^{t_{0},x_{0};w}) - \varphi(t_{0},x_{0})\Bigr] \ge 0. \label{EqP4.2}
\end{align}
Notice that $\varphi \in C_b^{1,2}([0, T] \times \mathbb{R}^d)$, then, using the It\^{o} formula, we can evaluate the difference between $\varphi(t_{0} + \delta t, X_{t_{0} + \delta t}^{t_{0},x_{0};w})$ and $\varphi(t_{0},x_{0})$ as follow
\begin{align}
\varphi(t_{0} + \delta t, X_{t_{0} + \delta t}^{t_{0},x_{0};w}) - &\varphi(t_{0},x_{0}) = \int_{t_{0}}^{t_{0} + \delta t} \Bigl[\dfrac{\partial}{\partial t} \varphi(s, X_{s}^{t_{0},x_{0};w}) + \mathcal{L}_{t}^{(\hat{u}_s,\alpha)} \varphi(s, X_{s}^{t_{0},x_{0};w}) \Bigr] d s \notag  \\
  & \quad + \int_{t_{0}}^{t_{0} + \delta t} D_x \varphi(s, X_{s}^{t_{0},x_{0};w}) \cdot \sigma(s, X_{s}^{t_{0},x_{0};w}, (\hat{u}_s,\alpha)) d B_s. \label{EqP4.3}
\end{align}
Moreover, if we substitute the above equation into \eqref{EqP4.2}, then we obtain
\begin{align}
\rho_{t_{0},t_{0} + \delta t}^{g_f} \Bigl[& \int_{t_{0}}^{t_{0} + \delta t} \Bigl[c_f\bigl(s, X_s^{t_{0},x_{0};w}, \alpha \bigr) + \dfrac{\partial}{\partial t} \varphi(s, X_{s}^{t_{0},x_{0};w}) + \mathcal{L}_{t}^{(\hat{u}_s,\alpha)} \varphi(s, X_{s}^{t_{0},x_{0};w}) \Bigr] d s \notag \\
& \quad + \int_{t_{0}}^{t_{0} + \delta t} D_x \varphi(s, X_{s}^{t_{0},x_{0};w}) \cdot \sigma(s, X_{s}^{t_{0},x_{0};w}, (\hat{u}_s,\alpha)) d B_s \Bigr]  \ge 0, \label{EqP4.4}
\end{align}
which amounts to solving the following BSDE
\begin{align}
&\tilde{Y}_{t_{0}}^{t_{0},x_{0};w} = \int_{t_{0}}^{t_{0} + \delta t} \Bigl[ \dfrac{\partial}{\partial t} \varphi(s, X_{s}^{t_{0},x_{0};w}) + \mathcal{L}_{t}^{(\hat{u},\alpha)} \varphi(s, X_{s}^{t_{0},x_{0};w}) \Bigr] d s \notag \\
&\,\  + \int_{t_{0}}^{t_{0} + \delta t} D_x \varphi(s, X_{s}^{t_{0},x_{0};w}) \cdot \sigma(s, X_{s}^{t_{0},x_{0};w}, (\hat{u}_s,\alpha)) d B_s \notag \\
&\,\,  + \int_{t_{0}}^{t_{0} + \delta t}g_f\bigl(s, X_s^{t_{0},x_{0};w}, \varphi(s, X_{s}^{t_{0},x_{0};w}), D_x \varphi(s, X_{s}^{t_{0},x_{0};w}) \cdot \sigma(s, X_{s}^{t_{0},x_{0};w}, (\hat{u}_s,\alpha))\bigr) ds \notag \\
& \quad \quad - \int_{t_{0}}^{t_{0} + \delta t} \tilde{Z}_s^{t_{0},x_{0};w} dB_s.  \label{EqP4.5}
\end{align}
From Lemma~\ref{L1}, the above BSDE admits unique solutions, i.e.,
\begin{align*}
\tilde{Z}_s^{t_{0},x_{0};w} = D_x \varphi(s, X_{s}^{t_{0},x_{0};w}) \cdot \sigma(s, X_{s}^{t_{0},x_{0};w}, (\hat{u}_s,\alpha)), \quad t_{0} \le s \le t_{0} + \delta t
\end{align*}
and
\begin{align*}
&\tilde{Y}_{t_{0}}^{t_{0},x_{0};w} = \int_{t_{0}}^{t_{0} + \delta t} \Bigl[\dfrac{\partial}{\partial t} \varphi(s, X_{s}^{t_{0},x_{0};w}) + \mathcal{L}_{t}^{(\hat{u}_s,\alpha)} \varphi(s, X_{s}^{t_{0},x_{0};w}) \notag \\
&\quad + g_f\bigl(s, X_{s}^{t_{0},x_{0};w}, \varphi(s, X_{s}^{t_{0},x_{0};w}), D_x \varphi(s, X_{s}^{t_{0},x_{0};w}) \cdot \sigma(s, X_{s}^{t_{0},x_{0};w}, (\hat{u}_s,\alpha))\bigr) \Bigr] d s.
\end{align*}
Further, if we substitute the above results in \eqref{EqP4.4}, we obtain
\begin{align}
&\int_{t_{0}}^{t_{0} + \delta t} \Bigl[\dfrac{\partial}{\partial t} \varphi(s, X_{s}^{t_{0},x_{0};w}) + \mathcal{L}_{t}^{(\hat{u},\alpha)} \varphi(s, X_{s}^{t_{0},x_{0};w}) \notag \\
&+ g_f\bigl(s, X_{s}^{t_{0},x_{0};w}, \varphi(s, X_{s}^{t_{0},x_{0};w}), D_x \varphi(s, X_{s}^{t_{0},x_{0};w}) \cdot \sigma(s, X_{s}^{t_{0},x_{0};w}, (\hat{u}_s,\alpha))\bigr)  \Bigr] d s \ge 0.  \label{EqP4.6}
\end{align}
Then, dividing the above equation by $\delta t$ and letting $\delta t \rightarrow 0$, we obtain
\begin{align*}
 \dfrac{\partial}{\partial t} \varphi(t_{0},x_{0}) &+ \mathcal{L}_{t}^{(\hat{u},\alpha)} \varphi(t_{0},x_{0}) \\
 &+ g_f\bigl(t_{0}, x_{0}, \varphi(t_{0},x_{0}), D_x \varphi(t_{0},x_{0}) \cdot \sigma(t_{0},x_{0}, (\hat{u}_{t_{0}},\alpha))\bigr) \ge 0. 
\end{align*}
Note that, since $\alpha \in V$ is arbitrary, we can rewrite the above condition as follow
\begin{align}
\dfrac{\partial}{\partial t_{0}} \varphi(t_{0},x_{0}) &+ \min_{\alpha \in V} \Bigl\{ \mathcal{L}_{t}^{(\hat{u},\alpha)} \varphi(t_{0},x_{0}) \notag\\
&+ g_f\bigl(t_{0}, x_{0}, \varphi(t_{0},x_{0}), D_x \varphi(t_{0},x_{0}) \cdot \sigma(t_{0},x_{0}, (\hat{u}_{t_{0}},\alpha))\bigr) \Bigr\} \ge 0, \label{EqP4.7} 
\end{align}
which attains its minimum in $V$ (which is a compact set in $\mathbb{R}^d)$. Thus, $V_f^{v}(\cdot,\cdot)$ is a viscosity subsolution of \eqref{Eq3.15}, with boundary condition $\varphi(T, x)=\Psi_f(T, x)$.

On the other hand, suppose that $\varphi \in C_b^{1,2}([0, T] \times \mathbb{R}^d)$ and assume that $\varphi \le V_f^{v}$ on $[0, T] \times \mathbb{R}^d$ and $\min_{(t,x)} \bigl[V_f^{v}(t,x) - \varphi(t,x)\bigr] =0$. Then, we consider a point $(t_{0},x_{0}) \in [0, T] \times \mathbb{R}^d$ so that $\varphi(t_{0},x_{0})= V_f^{v}(t_{0},x_{0})$ (i.e., a local minimum at $(t_{0},x_{0})$). Further, for a small $\delta t > 0$, Let $\tilde{v}_s$, which is restricted to $\Sigma_{[t_{0},t_{0} + \delta t]}$, be an $\epsilon \delta t$-optimal control for \eqref{Eq3.14} at $(t_{0},x_{0})$. Then, proceeding in this way as \eqref{EqP4.6}, with $w =(\hat{u}_s, \tilde{v}_s)$, we obtain the following
\begin{align}
&\int_{t_{0}}^{t_{0} + \delta t}  \Bigl[\dfrac{\partial}{\partial t} \varphi(s, X_{s}^{t_{0},x_{0};w}) + \mathcal{L}_{t}^{(\hat{u}_s,\tilde{v}_s)} \varphi(s, X_{s}^{t_{0},x_{0};w})  \notag \\
&+ g_f\bigl(s,X_{s}^{t_{0},x_{0};w}, \varphi(s, X_{s}^{t_{0},x_{0};w}), D_x \varphi(s, X_{s}^{t_{0},x_{0};w}) \cdot \sigma(s, X_{s}^{t_{0},x_{0};w}, (\hat{u}_s,\tilde{v}_s))\bigr) \Bigr] d s \le \epsilon \delta t.  \label{EqP4.8}
\end{align}
As a result of this, we also obtain the following
\begin{align}
&\int_{t_{0}}^{t_{0} + \delta t}  \min_{\alpha \in V} \Bigl\{\dfrac{\partial}{\partial t} \varphi(s, X_{s}^{t_{0},x_{0};w}) + \mathcal{L}_{t}^{(\hat{u},\alpha)} \varphi(s, X_{s}^{t_{0},x_{0};u})  \notag \\
&\,\, + g_f\bigl(s, X_{s}^{t_{0},x_{0};w}, \varphi(s, X_{s}^{t_{0},x_{0};w}), D_x \varphi(s, X_{s}^{t_{0},x_{0};w}) \cdot \sigma(s, X_{s}^{t_{0},x_{0};u}, (\hat{u}_s,\alpha))\bigr) \Bigr\} d s \le \epsilon \delta t.  \label{EqP4.9}
\end{align}
Note that the mapping 
\begin{align*}
(s, x, \alpha) \rightarrow \Bigl[\dfrac{\partial}{\partial t} \varphi(t, x) &+ \mathcal{L}_{t}^{(\hat{u}_t,\alpha)} \varphi(t, x) \\
                                                                                        &  + g_f\bigl(t, x, \varphi(t, x), D_x \varphi(t, x) \cdot \sigma(t, x, (\hat{u}_t,\alpha))\bigr) \Bigr]
\end{align*}
is continuous and, since $V$ is compact, then $s \rightarrow X_s^{t_{0},x_{0};w}$ is also continuous. As a result, the expression under the integral in \eqref{EqP4.9} is continuous. Further, if we divide both sides of \eqref{EqP4.9} by $\delta t$ and letting $\delta t \rightarrow 0$, then we obtain the following
\begin{align}
\dfrac{\partial}{\partial t_{0}} \varphi(t_{0},x_{0}) &+ \min_{\alpha \in V} \Bigl\{\mathcal{L}_{t}^{(\hat{u},\alpha)} \varphi(t_{0},x_{0}) \notag\\
&+ g_f\bigl(t_{0}, x_{0}, \varphi(t_{0},x_{0}), D_x \varphi(t_{0},x_{0}) \cdot \sigma(t_{0},x_{0}, (\hat{u}_{t_{0}},\alpha))\bigr) \Bigr\} \le \epsilon. \label{EqP4.10} 
\end{align}
Notice that, since $\epsilon$ is arbitrary, we conclude that $V_f^{v}(\cdot,\cdot)$ is a viscosity supersolution of \eqref{Eq3.13}, with boundary condition $\varphi(T,x)=\Psi_f(T, x)$. This completes the proof of Proposition~\ref{P4}. \qed
\end{proof}

\begin{remark} \label{R3.2}
Note that if \,$V_f^{v} \in C_b^{1,2}([0, T] \times \mathbb{R}^d)$, then such a solution also satisfies \eqref{Eq3.13} with boundary condition $V_f^{v}(T,x)=\Psi_f(T, x)$. Furthermore, using the {\it verification theorem}, one can also identify $V_f^{v}$ as the optimal value function.
\end{remark}

\begin{proposition} \label{P5}
Suppose that Proposition~\ref{P4} holds and let $\varphi \in C_b^{1,2}([0, T] \times \mathbb{R}^d)$ satisfy \eqref{Eq3.6} with $\varphi\bigl(T, x\bigr)=\Psi_f(T, x)$ for $x \in \mathbb{R}^d$. Then, $\varphi\bigl(t, x\bigr) \le V_f^{v}\bigl(t, x\bigr)$ for any control $v_{\cdot} \in \mathcal{V}_{[t,T]}$ with restriction to $\Sigma_{[t, T]}$ and for all $(t,x) \in [0, T] \times \mathbb{R}^d$. Furthermore, if an admissible control process $\hat{v}_{\cdot} \in \mathcal{V}_{[t,T]}$ exists, for almost all $(s, \Omega) \in [0, T] \times \Omega$, together with the corresponding solution $X_s^{t,x;\hat{w}}$, with $\hat{w}_s=(\hat{u}_s, \hat{v}_s)$, and satisfies
\begin{align}
 \hat{v}_s \in &\arginf_{v_{\cdot} \in \mathcal{V}_{[t,T]} \bigl \vert \Sigma_{[t, T]}} \Bigl\{ \mathcal{L}_{s}^{(u,v)} \varphi\bigl(s, X_s^{t,x;w}\bigr) \notag \\
 &~\underbrace{\hspace{0.5in} + \,g_f\bigl(s, X_s^{t,x;w}, \varphi \bigl(s, X_s^{t,x;w}\bigr), D_x\varphi \bigl(s, X_s^{t,x;w}\bigr) \cdot \sigma\bigl(s, X_s^{t,x;w}, (\hat{u}_s, v_s\bigr)\bigr)\bigr)\Bigr\}}_{\triangleq F(\hat{u}) \,\, \text{with} \,\, F \colon \hat{u}_{\cdot} \in \mathcal{U}_{[t,T]} \rightsquigarrow \hat{v}_{\cdot} \in \mathcal{V}_{[t,T]}} \label{Eq3.15}
 \end{align}
Then, $\varphi\bigl(t, x\bigr) = V_f^{\hat{v}}\bigl(t, x\bigr)$ for all $(t,x) \in [0, T] \times \mathbb{R}^d$.
\end{proposition}

\begin{proof}
Assume that $(t, x) \in [0, T] \times \mathbb{R}^d$ is fixed. For any $v_{\cdot} \in \mathcal{V}_{[t,T]}$, restricted to $\Sigma_{[t, T]}$, we consider a process $\kappa \bigl(s, X_{s}^{t,x;w}\bigr)$, with $w=(\hat{u}, v)$, for $s \in [t, T]$. Then, using It\^{o} integral formula, we can evaluate the difference between $\kappa \bigl(T, X_{T}^{t,x;w}\bigr)$ and $\kappa \bigl(t, x\bigr)$ as follow\footnote{Notice that $\kappa \bigl(t, x\bigr) \in C_b^{1,2}([0, T] \times \mathbb{R}^d).$}
\begin{align}
\kappa \bigl(T, X_{T}^{t,x;w}\bigr) - &\kappa \bigl(t, x\bigr) = \int_{t}^T \Bigl[\dfrac{\partial}{\partial t}\kappa\bigl(s, X_{s}^{t,x;w}\bigr) + \mathcal{L}_{t}^{(\hat{u}_s,v_s)} \kappa \bigl(s, X_{s}^{t,x;w}\bigr) \Bigr] d s \notag  \\
  & \quad + \int_{t}^T D_x \kappa \bigl(s, X_{s}^{t,x;w}\bigr) \cdot \sigma(s, X_{s}^{t,x;w}, (\hat{u}_s,v_s)) d B_s. \label{EqP5.1}
\end{align}
Using \eqref{Eq3.6}, we further obtain the following
\begin{align}
& \dfrac{\partial}{\partial t} \kappa\bigl(s, X_{s}^{t,x;w}\bigr) + \mathcal{L}_{t}^{(\hat{u}_s,v_s)} \kappa\bigl(s, X_{s}^{t,x;w}\bigr) \notag \\
& \quad + g_f\bigl(s, X_{s}^{t,x;w}, \kappa\bigl(s, X_{s}^{t,x;w}\bigr), D_x \kappa\bigl(s, X_{s}^{t,x;w}\bigr) \cdot \sigma(s, X_{s}^{t,x;w}, (\hat{u}_s,v_s))\bigr) \ge 0. \label{EqP5.2}
\end{align}
Furthermore, if we combine \eqref{EqP5.1} and \eqref{EqP5.2}, then we obtain
\begin{align}
\kappa \bigl(t, x\bigr) &\le \Psi_f\bigl(T, X_{T}^{t,x;w}\bigr) \notag \\
& \quad + \int_{t}^T g_f\bigl(s, X_{s}^{t,x;w}, \kappa(s, X_{s}^{t,x;w}), D_x \kappa(s, X_{s}^{t,x;w}) \cdot \sigma(s, X_{s}^{t,x;w}, (\hat{u}_s,v_s))\bigr) ds \notag \\
                                     & \quad  \quad -  \int_{t}^T D_x \kappa(s, X_{s}^{t,x;w}) \cdot \sigma(s, X_{s}^{t,x;w}, (\hat{u}_s,v_s)) d B_s. \label{EqP5.3}
\end{align}
Define $\tilde{Z}_{s}^{t,x;w} = D_x \kappa(s, X_{s}^{t,x;w}) \cdot \sigma(s, X_{s}^{t,x;u}, (\hat{u}_s,v_s))$, for $s \in [t, T]$, then $\kappa \bigl(t, x\bigr) \le \tilde{Y}_{t}^{t,x;w}$ follows, where $(\tilde{Y}_{\cdot}^{t,x;w}, \tilde{Z}_{\cdot}^{t,x;w})$ is a solution to BSDE in \eqref{Eq2.10}. As a result of this, we have
\begin{align*}
 \kappa \bigl(t, x\bigr) \le V_f^{v}\bigl(t, x\bigr).
\end{align*}
Moreover, if there exists at least one $\hat{v}$ satisfying \eqref{Eq3.15}. Then, for $v=\hat{v}$, the inequality in \eqref{EqP5.3} becomes an equality (i.e., $\kappa(t,x)=V_f^{\hat{v}}\bigl(t, x\bigr)$). Note that the corresponding path-wise solution $X_{s}^{t,x;\hat{w}}$, with $\hat{w}=(\hat{u}, \hat{v})$ and $\hat{v}=F(\hat{u})$, is progressively measurable, since $\hat{v}_{\cdot} \in \mathcal{V}_{[t,T]}$ is restricted to $\Sigma_{[t, T]}$. This completes the proof of Proposition~\ref{P5}. \qed
\end{proof}

\subsection{On the stochastic controllability} \label{S3.2}
As we have already mentioned in the previous sections (i.e., Section~\ref{S2} and Subsection~\ref{S3.1}), for a given {\it leader'}s risk-averse decision $\hat{u}_{\cdot} \in \mathcal{U}_{[t,T]}$, the risk-averse optimization in \eqref{Eq2.11} (or equation~\eqref{Eq2.10}) admits a unique solution $\hat{v}_{\cdot} = F(\hat{u}_{\cdot})$, which is restricted to $\Sigma_{[t, T]}$. However, the situation is more involved for the risk-averse optimization in \eqref{Eq2.12} (or equation~\eqref{Eq2.9}). Notice that it is not even clear that, for every $\xi^{Target} \in L^2(\Omega, \mathcal{F}_T, \mathbb{P}; \mathbb{R})$, there exist decision processes $u_{\cdot} \in \mathcal{U}_{[t,T]}$ and $A_{\cdot}^{t,x;w} \in \mathcal{S}^2\bigl(t, T; \mathbb{R} \bigr)$, with $w_{\cdot}=(u_{\cdot}, F(u_{\cdot}))$, such that
\begin{align}
 \hat{Y}_T^{t,x;w} \ge h(T, X_T^{t,x;w}), \label{Eq3.2.1}
\end{align}
\begin{align}
\int_t^T (\hat{Y}_s^{t,x;w} - h(s, X_s^{t,x;w})) d A_s^{t,x;w} = 0, \quad \text{and} \quad L_T = h(T, X_T^{t,x;w}), \label{Eq3.2.2}
\end{align}
for all $t \in [0, T]$. 

Moreover, verifying the above conditions amounted to solving the stochastic controllability type problem, which is indeed useful to describe the set of all acceptable risk-exposures, when $t=0$, vis-\'{a}-vis $\xi^{Target} \in L^2(\Omega, \mathcal{F}_T, \mathbb{P}; \mathbb{R})$, i.e.,
\begin{align}
\mathcal{A}_{0} = \Bigl\{\rho_{0, T}^{g_l} \bigl[\xi^{Target} \bigr] \ge L_T\, \bigl \vert \, &\int_0^T (\hat{Y}_s^{0,x;w} - h(s, X_s^{0,x;w})) d A_s^{0,x;w} = 0, \notag \\
&\quad  \,\, L_T = h(T, X_T^{t,x;w}), \quad t \in [0, T] \Bigr\}.  \label{Eq3.2.3}
\end{align}
In the following subsection, we provide additional results that provide conditions under which the problem in \eqref{Eq2.16} makes sense, if the {\it follower} is involved not only in minimizing his own accumulated risk-cost (in response to the decision of the {\it leader}) but also in minimizing that of the {\it leader}.

\subsection{On the risk-averse optimality condition for the {\it leader}} \label{S3.3}
In this subsection, we provide conditions under which the leader chooses its optimality risk-averse decision, whenever the {\it follower} responds optimally to the {\it leader'}s decision, i.e., $v=F(u)$, with restriction to $\Sigma_{[0, T]}$. Therefore, we suppose here that Proposition~\ref{P5} holds and, further, we will establish a two-way connection between the reflected-BSDE in \eqref{Eq2.9} and a probabilistic representation for the solution of related parabolic obstacle PDE problem.

Notice that, for each $t \in [0, T]$, the natural filtration of the Brownian motion $\{B_s - B_t, t \le s \le T\}$, augmented by the $\mathbb{P}$-null sets of $\mathcal{F}$, is denoted by $\{\mathcal{F}_s^t, t \le s < T\}$. Then, for each $(t, x) \in [0, T] \times \mathbb{R}^d$, with $w=(u,F(u))$, there exists a unique triple $\bigl(\hat{Y}^{t,x;w}, \hat{Z}^{t,x;w}, A^{t,x;w}\bigr)$ of $\bigl\{\mathcal{F}_s^t\bigr\}$ progressively measurable processes, which solves the following reflected-BSDE.
\begin{align}
&(i) \quad \mathbb{E} \int_t^T \bigl(\vert \hat{Y}_r^{t,x;w}\vert^2 + \vert \hat{Z}_r^{t,x;w} \vert^2\bigr) dr < \infty\notag  \\
&(ii) \quad \hat{Y}_s^{t,x;w} = \Psi_l(X_s^{t,x;w}) + \int_t^T g_l(r, X_r^{t,x;w},\hat{Y}_r^{t,x;w}, \hat{Z}_r^{t,x;w}) dr \notag  \\
& \quad \quad \quad \quad \quad \quad + A_T^{t,x;w} - A_r^{t,x;w} - \int_t^T \hat{Z}_r^{t,x;w} d B_r, \quad  t < s \le T,  \label{Eq3.3.1}\\
&(iii) \quad \hat{Y}_s^{t,x;w} \ge h(s, X_s^{t,x;w}), \quad t < s \le T \notag \\
&(iv) \quad \{A_s^{t,x;w}\} ~ \text{is increasing and continuous, and} \notag \\
&  \quad \quad \quad\quad \quad \quad \int_t^T (\hat{Y}_s^{t,x;w} - h(s, X_s^{t,x;w})) d A_s^{t,x;w} = 0. \notag 
\end{align}
More precisely, we consider the following related parabolic obstacle PDE problem. Then, roughly speaking, its solution is a function $\varphi \colon [0, T] \times \mathbb{R}^d \rightarrow \mathbb{R}$ which satisfies
\begin{align}
\min \Bigl(\varphi(t,x) - h(t,x), &-\frac{\partial \varphi}{\partial t}(t,x) - \inf_{u \in U} \Bigl\{ \mathcal{L}_t^{(u,F(u))} \varphi(t,x) \notag \\ 
& \quad + g_l(t, x, \varphi(t,x), D_{x} \varphi(t,x) \cdot \sigma(s, x, (u, F(u))))  \Bigr\} \Bigr) = 0, \notag \\
& \quad\quad\quad\quad\quad\quad (t,x) \in (0,T) \times \mathbb{R}, \notag\\
&\varphi(T,x) = \Psi_l(x), \quad x \in \mathbb{R}^d. \label{Eq3.3.2}
\end{align}

Hence, we consider such a solution for \eqref{Eq3.3.2} in the sense of viscosity. For the sake of convenience, we also provide the definition of viscosity solutions for the above parabolic obstacle PDE problem (cf. Definition~\ref{Df3}).
\newpage
\begin{definition} \label{Df4} ~\\ \vspace{-0.2in}
\begin{enumerate} [(a)]
\item $\varphi \in C([0,T] \times \mathbb{R}^d)$ is said to be a viscosity subsolution of \eqref{Eq3.3.2} if $\varphi(T,x) \le \Psi_l(x)$, $x \in \mathbb{R}^d$, and at any point $(t,x) \in (0,T) \times \mathbb{R}^d$, 
\begin{align*}
\min \Bigl(\varphi(t,x) - h(t,x), & -\frac{\partial \psi}{\partial t}(t,x) - \inf_{u \in U} \Bigl\{ \mathcal{L}_t^{(u,F(u))} \psi(t,x) \\ 
& \quad + g_l(t, x, \varphi(t,x), D_{x} \psi(t,x) \cdot \sigma(s, x, (u, F(u))))  \Bigr\} \Bigr) \le 0.
\end{align*}
In other words, at any point $(t,x)$, where $\varphi(t,x) > h(t,x)$
\begin{align*}
 -\frac{\partial \psi}{\partial t}(t,x) &- \inf_{u \in U} \Bigl\{ \mathcal{L}_t^{(u,F(u))} \psi(t,x) \\ 
& \quad + g_l(t, x, \varphi(t,x), D_{x} \psi(t,x) \cdot \sigma(s, x, (u, F(u)))) \Bigr\} \le 0.
\end{align*}
\item $\varphi \in C([0,T] \times \mathbb{R}^d)$ is said to be a viscosity supersolution of \eqref{Eq3.3.2} if $\varphi(T,x) \ge \Psi_l(x)$, $x \in \mathbb{R}^d$, and at any point $(t,x) \in (0,T) \times \mathbb{R}^d$, 
\begin{align*}
\min \Bigl(\varphi(t,x) - h(t,x), & -\frac{\partial \psi}{\partial t}(t,x) - \inf_{u \in U} \Bigl\{ \mathcal{L}_t^{(u,F(u))} \psi(t,x) \\ 
& \quad + g_l(t, x, \varphi(t,x), D_{x} \psi(t,x) \cdot \sigma(s, x, (u, F(u))))  \Bigr\} \Bigr) \ge 0.
\end{align*} 
\item $\varphi \in C([0,T] \times \mathbb{R}^d)$ is said to be a viscosity solution of \eqref{Eq3.3.2} if is both a viscosity sub- and supersolution.
\end{enumerate}
\end{definition}  

\begin{lemma} \label{L2}
For $(t,x) \in [0,T] \times \mathbb{R}^d$, define 
\begin{align}
\varphi(t,x) \triangleq \hat{Y}_t^{t,x;w}. \label{Eq3.3.3}
\end{align} 
Then, $\varphi \in C([0,T] \times \mathbb{R}^d)$ is a deterministic quantity.
\end{lemma}

\begin{proof}
We define $\hat{Y}_s^{t,x;w}$ for all $s \in [0,T]$ by choosing $\hat{Y}_s^{t,x;w}=\hat{Y}_t^{t,x;w}$ for $0 \le s \le t$. It is suffices to show that whenever $(t_n, x_n) \rightarrow (t,x)$
\begin{align}
\mathbb{E} \Biggl\{\sup_{0 \le s \le T} \Bigl\vert \hat{Y}_s^{t_n,x_n;w} - \hat{Y}_s^{t,x;w}\Bigr\vert^2 \Biggr \} \rightarrow 0. \label{Eq3.3.4}
\end{align} 
Indeed, this will show that
\begin{align}
(s,t,x)  \rightarrow \hat{Y}_s^{t,x;w} \label{Eq3.3.5}
\end{align} 
is mean-square continuous, and so is  
\begin{align}
(t,x)  \rightarrow \hat{Y}_t^{t,x;w}. \label{Eq3.3.6}
\end{align} 
But $\hat{Y}_t^{t,x;w}$ is deterministic, since it is $\mathcal{F}_t^t$-measurable.

Furthermore, note that \eqref{Eq3.3.4} is a consequence of Proposition~\ref{AP2} (see the Appendix section) and the following convergences as $n \rightarrow \infty$:  
\begin{align*}
\mathbb{E} \Bigl\{\bigl\vert \Psi_l(X_T^{t,x;w}) &- \Psi_l(X_T^{t_n,x_n;w}) \bigr\vert^2 \Bigr\} \rightarrow 0 \\
\mathbb{E} \Bigl\{\sup_{0 \le s \le T} \bigl\vert h(s, X_s^{t,x;w}) &- h(s, X_s^{t_n,x_n;w})\bigr\vert^2 \Bigr \} \rightarrow 0 \\
\mathbb{E} \Bigl\{ \int_0^T \bigl\vert \mathbf{1}_{[t,T]} g_l(s, X_s^{t,x;w}, \hat{Y}_s^{t,x;w}, \hat{Z}_s^{t,x;w}) & \\
- \mathbf{1}_{[t_n,T]}g_l(s, X_s^{t_n,x_n;w},& \hat{Y}_s^{t_n,x_n;w}, \hat{Z}_s^{t_n,x_n;w})\bigr\vert^2 ds \Bigr \} \rightarrow 0,
\end{align*} 
that follow from Assumptions~\ref{AS1} and \ref{AS2}, and the growth conditions of $g_l$, $\Psi_l$ and $h$. This completes the proof of Lemma~\ref{L2}. \qed
\end{proof}

\begin{proposition} \label{P6}
Suppose that Lemma~\ref{L2} holds, then $\varphi$ is a viscosity solution of the parabolic obstacle PDE in \eqref{Eq3.3.2}.
\end{proposition}

\begin{proof}
In order to prove the above proposition, we will use an approximation for the reflected-BSDE of \eqref{Eq2.9} by penalizing. For each $(t,x) \in [0, T] \times \mathbb{R}^d$, $n \in \mathbb{N}$, let $\bigl\{(\hat{Y}_{n,s}^{t,x;w}, \hat{Z}_{n,s}^{t,x;w}), t \le s \le T\bigr\}$ denote the solution of the following approximated reflected-BSDE
\begin{align*}
&\hat{Y}_{n,s}^{t,x;w} = \Psi_l(X_s^{t,x;w}) + \int_t^T g_l(r, X_r^{t,x;w}, \hat{Y}_{n,r}^{t,x;w}, \hat{Z}_{n,r}^{t,x;w}) dr  \\
& \quad \quad \quad + n \int_t^T \bigl(\hat{Y}_{n,r}^{t,x;w} - h(r, X_r^{t,x;w})\bigr)^{-} - \int_t^T \hat{Z}_{n,r}^{t,x;w} d B_r, \quad  t < s \le T. \\ 
\end{align*}
Then, from \cite{RevY94}, it is clear that
\begin{align*}
\varphi_n(t,x) \triangleq \hat{Y}_{n,t}^{t,x;w}, \quad 0 \le t \le T, \quad x \in \mathbb{R}^d,
\end{align*} 
is also the viscosity solution of the following parabolic PDE
\begin{align*}
  &\frac{\partial \varphi_n}{\partial t}(t,x) + \inf_{u \in U} \Bigl\{ \mathcal{L}_t^{(u,F(u))} \varphi_n(t,x) \\
 & \quad + \hat{g}_l(t, x, \varphi_n(t,x), D_{x} \varphi_n (t,x) \cdot \sigma(t,x,(u,F(u)))) \Bigr\} = 0,  \,\, (t,x) \in (0,T) \times \mathbb{R}\\
& \quad \quad  \varphi_n(T,x) = \Psi_l(x), \quad x \in \mathbb{R}^d, 
\end{align*}
where 
\begin{align*}
&\hat{g}_l(t, x, \varphi(t,x), D_{x} \varphi(t,x) \cdot \sigma(t,x,(u,F(u)))) \\
& \quad \quad = g_l(t, x, \varphi(t,x), D_{x} \varphi (t,x) \cdot \sigma(t,x,(u,F(u)))) - \bigl(\varphi(t,x) - h(t,x)\bigr)^{-}.
\end{align*}
Notice that, for each $ t \in [0,T]$, $x \in \mathbb{R}^d$, we have
\begin{align*}
\varphi_n(t,x) \uparrow \varphi(t,x)  \quad \text{as} \quad n \rightarrow \infty.
\end{align*} 
Since $\varphi_n$ and $\varphi$ are continuous, it follows from Dini's theorem that the above convergence is uniform on compacts.

Next, we show that $\varphi$ is a subsolution of \eqref{Eq3.3.2}. Let $(t,x)$ be a point at which $\varphi(t,x) > h(t,x)$, and let   

From Lemma~6.1 in \cite{CraIL92}, there exists sequences
\begin{align*}
n_j \rightarrow +\infty \quad \text{and} \quad (t_j, x_j) \rightarrow (t,x)
\end{align*}
such that 
\begin{align*}
\Bigl(\frac{\partial \psi_j}{\partial t}, D_{x} \psi_j, D_{x}^2 \psi_j \Bigr) \rightarrow \Bigl(\frac{\partial \psi}{\partial t}, D_{x} \psi, D_{x}^2 \psi \Bigr)
\end{align*}
but for any $j$
\begin{align*}
 -\frac{\partial \psi_j}{\partial t}(t_j, x_j) &-  \inf_{u \in U} \Bigl\{\mathcal{L}_t^{(u,F(u))} \psi_j(t_j, x_j) \\
                                             &- g_l(t_j, x_j, \varphi_{n_j}(t_j, x_j), D_{x} \psi_j(t_j, x_j) \cdot \sigma(t_j, x_j, (u,F(u))))\Bigr\}  \\
                                                     & \quad \quad - n_j\bigl(\varphi_{n_j}(t_j, x_j) - h(t_j, x_j) \bigr)^{-} \le 0.
\end{align*}
From the assumption that $\varphi(t,x) > h(t,x)$ and the uniform convergence of $\varphi_n$, it follows that for $j$ large enough $\varphi_{n_j}(t_j, x_j) > h(t_j, x_j)$; hence taking the limit as $n_j \rightarrow +\infty$ in the above inequality yields
\begin{align*}
 -\frac{\partial \psi}{\partial t}(t,x) &- \inf_{u \in U} \Bigl\{\mathcal{L}_t^{(u,F(u))} \psi(t, x) \\
                                             &- g_l(t, x, \varphi(t, x), D_{x} \psi(t, x) \cdot \sigma(t, x, (u,F(u))))\Bigr\} \le 0
\end{align*}
and we have proved that $\varphi$ is a subsolution of \eqref{Eq3.3.2}.

Then, we conclude the proof by showing that $\varphi$ is a supersolution of \eqref{Eq3.3.2}. Let $(t,x)$ be an arbitrary point in $[0,T] \times \mathbb{R}^d$ and . We already know that $\varphi(t,x) \ge h(t,x)$. By the same argument as above, there exist sequences:
\begin{align*}
n_j \rightarrow +\infty \quad \text{and} \quad (t_j, x_j) \rightarrow (t,x)
\end{align*}
such that 
\begin{align*}
\Bigl(\frac{\partial \psi_j}{\partial t}, D_{x} \psi_j, D_{x}^2 \psi_j \Bigr) \rightarrow \Bigl(\frac{\partial \psi}{\partial t}, D_{x} \psi, D_{x}^2 \psi \Bigr)
\end{align*}
but for any $j$
\begin{align*}
 -\frac{\partial \psi_j}{\partial t}(t_j, x_j) &-  \inf_{u \in U} \Bigl\{\mathcal{L}_t^{(u,F(u))} \psi_j(t_j, x_j) \\
                                             &- g_l(t_j, x_j, \varphi_{n_j}(t_j, x_j), D_{x} \psi_j(t_j, x_j) \cdot \sigma(t_j, x_j, (u,F(u))))\Bigr\}  \\
                                                     & \quad \quad - n_j\bigl(\varphi_{n_j}(t_j, x_j) - h(t_j, x_j) \bigr)^{-} \ge 0.
\end{align*}
Hence, we have
\begin{align*}
 -\frac{\partial \psi_j}{\partial t}(t_j, x_j) &-  \inf_{u \in U} \Bigl\{\mathcal{L}_t^{(u,F(u))} \psi_j(t_j, x_j) \\
                                             &- g_l(t_j, x_j, \varphi_{n_j}(t_j, x_j), D_{x} \psi_j(t_j, x_j) \cdot \sigma(t_j, x_j, (u,F(u))))\Bigr\} 
\end{align*}
and taking the limit as $n_j \rightarrow +\infty$, then we conclude that
\begin{align*}
 -\frac{\partial \psi}{\partial t}(t,x) &- \inf_{u \in U} \Bigl\{\mathcal{L}_t^{(u,F(u))} \psi(t, x) \\
                                             &- g_l(t, x, \varphi(t, x), D_{x} \psi(t, x) \cdot \sigma(t, x, (u,F(u))))\Bigr\} \ge 0.
\end{align*}
This completes the proof of Proposition~\ref{P6}. \qed
\end{proof}

We conclude this subsection with the following proposition, which provides a condition for the {\it leader} to have an optimal risk-averse decision.
\begin{proposition} \label{P7}
Let $\varphi \in C([0, T] \times \mathbb{R}^d)$ be a viscosity solution for the parabolic obstacle PDE in \eqref{Eq3.3.2}, with boundary condition $\varphi(T, x)=\Psi_l(T, x)$, for $x \in \mathbb{R}^d$. Then, $\varphi(t, x) \le V_l^{u}\bigl(t, x\bigr)$ for some $u_{\cdot} \in \mathcal{U}_{[t,T]}$, with restriction to $\Sigma_{[t, T]}$, and for all $(t,x) \in [0, T] \times \mathbb{R}^d$. Furthermore, if an admissible decision process $u_{\cdot}^{\ast} \in \mathcal{U}_{[t,T]}$ exists, for almost all $(s, \Omega) \in [0, T] \times \Omega$, together with the corresponding solution $X_s^{t,x;w^{\ast}}$, with $w_{\cdot}^{\ast}=(u_{\cdot}^{\ast}, F(u_{\cdot}^{\ast}))$, and satisfies
\begin{align}
 &u_s^{\ast} \in \arginf_{u_{\cdot} \in \mathcal{U}_{[t,T]} \bigl \vert \Sigma_{[t, T]}} \Bigl\{\mathcal{L}_{s}^{(u, F(u))} \varphi(s, X_s^{t,x;w}) \notag \\
 & \quad + g_l(s, X_s^{t,x;w}, \varphi (s, X_s^{t,x;w}), D_x\varphi(s, X_s^{t,x;w}) \cdot \sigma(s, X_s^{t,x;w}, (u_s, F(u_s)))) \Bigr\}. \label{Eq3.17}
 \end{align}
Then, $\varphi\bigl(t, x\bigr) = V_l^{u^{\ast}}\bigl(t, x\bigr)$ for all $(t,x) \in [0, T] \times \mathbb{R}^d$.
\end{proposition}

\begin{proof}
The proof is similar to that of Proposition~\ref{P5}, except that we require a unique solution set $\bigl\{X_{\cdot}^{t,x;u}, (\hat{Y}_{\cdot}^{t,x;w}, \hat{Z}_{\cdot}^{t,x;w}, A_{\cdot}^{t,x;w}), (\tilde{Y}_{\cdot}^{s,x;w}, \tilde{Z}_{\cdot}^{t,x;w})\bigr\}$ for the FBSDEs in \eqref{Eq2.3}, \eqref{Eq2.10} and the reflected-BSDE in \eqref{Eq2.9} on $\bigl(\Omega, \mathcal{F}, P, \mathcal{F}^t\bigr)$ for every initial condition $(t,x) \in [0,T] \times \mathbb{R}^d$.
\end{proof}

\section{Further remarks} \label{S4}
In this section, we further comment on the implication of our result in assessing the influence of the {\it leader'}s decisions on the risk-averseness of the {\it follower} in relation to the direction of {\it leader-follower} information flow. Note that the statement of Proposition~\ref{P5} is implicitly accounted in Proposition~\ref{P7}. That is, for $s \in [t, T]$, the risk-averseness of the {\it follower}, with restriction to $\Sigma_{[t,T]}$,
\begin{align*}
 &v_s^{\ast} \in \arginf_{v_{\cdot} \in \mathcal{V}_{[t,T]} \bigl \vert \Sigma_{[t, T]}} \Bigl\{\mathcal{L}_{s}^{(F^{-1}(v), v)} \varphi(s, X_s^{t,x;w}) \notag \\
 & \quad + g_f(s, X_s^{t,x;w}, \varphi (s, X_s^{t,x;w}), D_x\varphi(s, X_s^{t,x;w}) \cdot \sigma(s, X_s^{t,x;w}, (F^{-1}(v_s), v_s))) \Bigr\},
 \end{align*}
with $w_{\cdot}=(F^{-1}(v_{\cdot}),v_{\cdot})$ and $F^{-1} \colon v_{\cdot} \in \mathcal{V}_{[t,T]} \rightsquigarrow u_{\cdot} \in \mathcal{U}_{[t,T]}$, is a subproblem in \eqref{Eq3.17}. 

On the other hand, the risk-averse decision of the {\it leader}
\begin{align*}
 &u_s^{\ast} \in \arginf_{u_{\cdot} \in \mathcal{U}_{[t,T]} \bigl \vert \Sigma_{[t, T]}} \Bigl\{\mathcal{L}_{s}^{(u, F(u))} \varphi(s, X_s^{t,x;w}) \notag \\
 & \quad + g_l(s, X_s^{t,x;w}, \varphi (s, X_s^{t,x;w}), D_x\varphi(s, X_s^{t,x;w}) \cdot \sigma(s, X_s^{t,x;w}, (u_s, F(u_s)))) \Bigr\},
 \end{align*}
that is implicitly conditioned by the {\it leader'}s decision $u$ and that of the {\it follower's} decision $v = F(u)$. As a result of this, the {\it follower} is involved not only in minimizing his own accumulated risk-cost (in response to the risk-averse decsion of the {\it leader}) but also in minimizing that of the {\it leader's} accumulated risk-cost.
Hence, such an inherent interaction, due to the {\it nature of the problem}, constitutes a constrained information flow between the {\it leader} and that of the {\it follower}, in which the {\it follower} is required to respond optimally, in the sense of {\it best-response correspondence} to the risk-averse decision of the {\it leader}.

\section*{Appendix}
In this section, we provide additional results (whose proofs are adaptations of \cite{ElKKPQ97}) that are related to the solutions of the reflected-BSDE in \eqref{Eq3.3.1}.

\begin{proposition}\label{AP1}
For $(t, x) \in [0, T] \times \mathbb{R}^d$, let $\bigl(\hat{Y}_{s}^{t,x;w}, \hat{Z}_{s}^{t,x;w}, A_{s}^{t,x;w}\bigr)_{t \le s \le T}$ be the solution of the reflected BSDE satisfying (iii) and (iv) in \eqref{Eq3.3.1}. Then, for each $t \in [0, T]$,
\begin{align}
& \hat{Y}_{t}^{t,x;w} = \esssup_{\tau \in \mathscr{T}_t} \mathbb{E} \Bigl\{L_{\tau} \mathbf{1}_{\tau < T} + \xi^{Target} \mathbf{1}_{\tau = T} \notag \\
 & \,\, + \int_t^{\tau}  g_l(s, X_s^{t,x;w}, \varphi (s, X_s^{t,x;w}), D_x\varphi(s, X_s^{t,x;w}) \cdot \sigma(s, X_s^{t,x;w}, (u_s, F(u_s)))) ds \, \Bigl \vert \, \mathcal{F}_t \Bigr\}, \label{EqA.1}
\end{align}
where $\mathscr{T}$ is the set of all stopping times dominated by T, and
\begin{align*}
\mathscr{T}_t = \bigl\{ \tau \in \mathscr{T} \,\vert\, t \le \tau \le T \bigr\}.
\end{align*}
\end{proposition}
Then, we have the following proposition, whose proof depends on the above proposition and use of the Gronwall's lemma and Burkholder-Davis-Gundy's inequality.
\begin{proposition}\label{AP2}
Suppose that $(g_l, \xi^{Target}, L)$ and $(\bar{g}_l, \bar{\xi}^{Target}, \bar{L})$ satisfies Assumption~\ref{AS2}. Let $(\hat{Y}_{\cdot}^{0,x;w}, \hat{Z}_{\cdot}^{0,x;w}, A_{\cdot}^{0,x;w})$ and $(\bar{Y}_{\cdot}^{0,x;w}, \bar{Z}_{\cdot}^{0,x;w}, \bar{A}_{\cdot}^{0,x;w})$ be solutions of reflected-BSDEs associated with $(g_l, \xi^{Target}, L)$ and $(\bar{g}_l, \bar{\xi}^{Target}, \bar{L})$, respectively. Define
\begin{align*}
&\Delta \xi^{Target}=\xi^{Target} - \bar{\xi}^{Target},&  &\Delta g_l=g_l - \bar{g}_l,& \\
&\Delta L = L - \bar{L},&  &\Delta A^{0,x;w} = A^{0,x;w} - \bar{A}^{0,x;w},& \\
&\Delta Y^{0,x;w}=Y^{0,x;w} - \bar{Y}^{0,x;w},&  &\Delta Z^{0,x;w} = Z^{0,x;w} - \bar{Z}^{0,x;w}.& 
\end{align*}
Then, there exists a constant $\gamma$ such that
\begin{align*}
 & \mathbb{E} \Bigl\{ \sup_{0 \le t \le T} \vert \Delta Y_t^{0,x;w} \vert^2 + \int_0^T \vert \Delta Z_t^{0,x;w}\vert^2 dt + \Delta \vert A_T^{0,x;w} \vert^2 \Bigr\} \\
  & \quad\quad \le \gamma  \mathbb{E} \Bigl\{ \vert \Delta \xi_t^{Target} \vert^2 + \int_0^T \vert \Delta g_l(t, X_t^{0,x;w}, \hat{Y}_t^{0,x;w}, \hat{Z}_t^{0,x;w}) \vert^2 dt\Bigr\} \\
   & \quad\quad\quad\quad+ \gamma \Bigl[\mathbb{E} \Bigl\{  \sup_{0 \le t \le T} \vert A_t^{0,x;w} \vert^2\Bigr\} \Bigr]^{\frac{1}{2}} \bigl[\Gamma_T \bigr]^{\frac{1}{2}}, 
\end{align*}
where
\begin{align*}
 \Gamma_T = \mathbb{E} \Bigl\{ (\xi_t^{Target})^2 &+ \int_0^T g_l^2(t, x, 0, 0)dt + \sup(L_t^{+})^2 \\
  &  \quad + (\bar{\xi}_t^{Target})^2 + \int_0^T \bar{g}_l^2(t, x, 0, 0) dt +\sup(\bar{L}_t^{+})^2\Bigr\}.
\end{align*}
\end{proposition}

Then, we have the following uniqueness result from the Proposition~\ref{AP2} with $g_l=\bar{g}_l$, $L=\bar{L}$ and $\xi^{Target}=\bar{\xi}^{Target}$. 
\begin{corollary}\label{AC1} 
Under Assumption~\ref{AS2}, there exists at most one progressively measurable triple $\bigl(\hat{Y}_s^{t,x;u}, \hat{Z}_s^{t,x;u}, A_s^{t,x;u}\bigr)_{t \le s \le T}$ satisfying (iii) and (v) in \eqref{Eq3.3.1}.
\end{corollary}

\begin{remark} \label{RA.1}
Note that, for $(t, x) \in [0, T] \times \mathbb{R}^d$, if $\bigl(\hat{Y}_s^{t,x;u}, \hat{Z}_s^{t,x;u}, A_s^{t,x;u}\bigr)_{t \le s \le T}$ satisfying (iii) and (v) in \eqref{Eq3.3.1}. Then, we have 
\begin{align*}
 A_{T}^{t,x;w} - A_{t}^{t,x;w} = \sup_{t \le s \le T} \Bigl\{\xi^{Target} &+ \int_s^T  g_l(\tau, X_{\tau}^{t,x;w}, \hat{Y}_{\tau}^{t,x;w}, \hat{Z}_{\tau}^{t,x;w}) d \tau \\
 & - \int_s^T \hat{Z}_{\tau}^{t,x;w} dB_{\tau}  - L_{\tau}\Bigr\}^{-},
\end{align*}
for each $s \in [t, T]$.
\end{remark}

\end{document}